\documentclass[12pt]{article}
\usepackage{fullpage}

\usepackage{bm}
\usepackage{amsmath,amssymb}
\usepackage{amsthm}
\usepackage{wrapfig}
\usepackage{epsfig}
\usepackage{graphicx}
\usepackage{makeidx}
\usepackage{multicol}
\usepackage{tikz}
\usepackage{algorithm}
\usepackage{algpseudocode}
\usepackage{soul}

\newcommand{\N}{{\cal N}}
\newcommand{\M}{{\cal M}}
\newcommand{\D}{{\cal D}}

\newtheorem{thm}{Theorem}
\newtheorem{lem}[thm]{Lemma}

\graphicspath{{../../figures/}}
\newcommand\rev[1]{\textcolor{black}{#1}}

\newcommand{\bs}[1]{\boldsymbol{#1}}
\newcommand{\bmu}{\bs{\mu}}

\begin{document}

\title{{Reduced-basis method for the iterative solution of parametrized symmetric positive-definite linear systems}}

\author{Ngoc-Cuong Nguyen\footnote{MIT Department of Aeronautics and Astronautics, 77 Massachusetts Ave., Cambridge, MA 02139, USA. Email: \texttt{cuongng@mit.edu}.  This author was partially supported by the Air Force Office of Scientific Research (FA9550-16-1-0214 and  FA9550-15-1-0276), the National Aeronautics and Space Administration (NASA NNX16AP15A) and Pratt \& Whitney.}, \, Yanlai Chen\footnote{Department of Mathematics, University of Massachusetts Dartmouth, 285 Old Westport Road, North Dartmouth, MA 02747, USA. Email: {\tt{yanlai.chen@umassd.edu}}. This author was partially supported by National Science Foundation grant DMS-1719698. \newline ${   } \quad \ {  }$ We would like to thank Professor Jaime Peraire at MIT and Mr. Pablo Fernandez at MIT for helpful discussion.}}

\date{\empty}

\maketitle

\begin{abstract}
We present a class of reduced basis (RB) methods for the iterative solution of parametrized symmetric positive-definite (SPD) linear systems. The essential ingredients are a Galerkin projection of the underlying parametrized system onto a reduced basis space to obtain a reduced system; an adaptive greedy algorithm to efficiently determine sampling parameters and associated basis vectors; an offline-online computational procedure \rev{and a multi-fidelity approach} to decouple the construction and application phases of the reduced basis method; and solution procedures to employ the reduced basis approximation as a {\em stand-alone iterative solver} or as a {\em preconditioner} in the conjugate gradient method. We present numerical examples to demonstrate the performance of the proposed methods in comparison with multigrid methods. Numerical results show that, \rev{when applied to solve linear systems resulting from discretizing the Poisson's equations, the speed of convergence of our methods matches or surpasses that of the multigrid-preconditioned conjugate gradient method}, while their computational cost per iteration is significantly smaller \rev{providing a feasible alternative when the multigrid approach is out of reach due to timing or memory constraints for large systems}. 
\rev{Moreover, numerical results verify that this new class of reduced basis methods, when applied as a stand-alone solver or as a preconditioner, is capable of achieving the accuracy at the level of the {\em truth approximation} which is far beyond the RB level.}
\end{abstract}

\newpage

\section{Introduction}

In this paper, we present a class of subspace iterative methods for solving  parametrized SPD linear systems of the form: 
\begin{equation}
\label{eq1}
 A_{\N}( \mu)  x_\N(\mu) =  f_\N( \mu) ,
\end{equation}
where $A_\N( \bmu)$ is an SPD parameter-dependent matrix of dimension $\N \times \N$ and $f_\N( \bmu)$ is a parameter-dependent vector of dimension $\N$. Here the parameter vector $ \bmu = (\mu_1, \mu_2, \ldots, \mu_P)$  resides in a parameter space $\D \in \mathbb{R}^P$.  The parametrized linear  system (\ref{eq1}) often arises in the context of parametric analysis,  engineering design and optimization, and statistics. The development of  fast and reliable methods for solving the system (\ref{eq1}) with {\em many queries} of the parameter vector in the parameter space is \rev{of} significant interest due to a wide variety of applications in engineering and science.

Classical iterative methods such as Jacobi, Richardson, and Gauss-Seidel methods have been used to solve \rev{the linear system} \eqref{eq1}. A generalization of Gauss-Seidel method led to the successive over-relaxation (SOR) method devised by Young and Frankel \cite{Young1954}. An alternative to classical iterative methods are Krylov subspace methods. The conjugate gradient (CG) method developed by Hestenes and Stiefel \cite{Hestenes1952} is well suited for solving symmetric positive-definite linear systems. Other Krylov methods for linear systems include CGS \cite{Sonneveld1989}, BiCGSTAB \cite{Vorst2006}, MINRES \cite{Paige1975}, GMRES \cite{sasc86}, and QMR \cite{Freund1991}, to name a few. Adopting ideas from Nesterov methods \cite{Nesterov2004a,Nesterov2012,Nesterov2004} for convex optimization, accelerated residual methods are recently developed in \cite{Nguyen2018} to solve linear and nonlinear systems. Multigrid methods \cite{brandt1977multi,briggs2000multigrid,hackbusch2013multi,Mavriplis1998a} have also been widely used as an iterative solver or as a preconditioner for preconditioned Krylov methods.  For symmetric positive definite systems, fundamental theoretical convergence results are established, and efficient multigrid solvers have been developed.

General-purpose iterative solvers can be computationally prohibitive for solving the system (\ref{eq1}) repeatedly over a large number of parameter samples. As a result, there exist iterative methods that exploit the parameter dependence of the system in some particular ways. Because Krylov subspaces are invariant for shifted matrices,  efficient Krylov methods \cite{Darnell2008,Frommer1998} have been developed to simultaneously solve shifted linear systems. These methods can be easily extended to polynomial dependence on a single parameter by means of  linearization \cite{Grammont2011,Gu2005,simoncini2002numerical}. Exploiting the fact that a sequence of linear systems $A_\N(\bmu_j) x_\N(\bmu_j) = f_\N(\bmu_j), 1 \le j \le J$, have some important similarities, Krylov subspace recycling \cite{kilmer2006recycling,parks2006recycling} has been proposed as a means to speed up Krylov methods for sequence of linear systems. Instead of discarding the Krylov space generated when solving a linear system, one can judiciously select a subspace and use it to reduce the number of iterations for solving the next system. However,  it is not completely clear how subspace selection affects convergence. Low-rank  Krylov subspace methods \cite{Kressner2011} combine the sequence of linear systems, $A_\N(\bmu_j) x_\N(\bmu_j) = f_\N(\bmu_j), 1 \le j \le J$, into one large linear system and exploit the low-rank structure of the resulting linear system to solve it efficiently.

The reduced basis (RB) method \cite{Barrault2004,Boyaval2010,Buffa2012a,ARCME,Chen2010a,Eftang2010a,Eftang2012a,Grepl2007,prudhomme02:_reliab_real_time_solut_param,VeroyPatera05, HesthavenRozzaStammBook, BennerCohenOhlbergerWillcoxBook} has been widely used to enable fast and reliable approximation of the solution of the parametrized linear system (\ref{eq1}) arising from the spatial discretization of a parametrized linear partial differential equation. The RB method exploits the fact that the solution $x_\N(\bmu)$ resides in a {\em low-dimensional manifold} $\M$ shaped by the parameter dependence of the matrix $A_\N(\bmu)$ and the vector $f_\N(\bmu)$ with respect to $\bmu \in \D$. As $\bmu$ varies in $\D$, $x_\N(\bmu)$ also varies in $\M$. Instead of searching $x_\N(\bmu)$ in an $\N$-dimensional space $\mathbb{R}^\N$, the RB method looks for an approximation $\widehat{x}_\N(\bmu)$ in an $N$-dimensional space $W_N$. This RB space $W_N$ is a subspace of the low-dimensional manifold $\M$. The RB method relies on inexpensive \rev{but} rigorous {\em a posteriori} error estimation to certify the error $\|x_\N(\bmu) - \widehat{x}_{\N}(\bmu)\|$ and construct $W_N$ through the greedy sampling procedure. However, in constructing $W_N$, the RB method still relies on standard solution methods to solve the system (\ref{eq1}) at $N$ particular parameter vectors in the sample set $S_N = \{\bmu_1, \bmu_2,\ldots, \bmu_N\}$.

Our goal in this paper is to develop RB methods that iteratively solve the parametrized system (\ref{eq1}) for $x_\N(\bmu)$, \rev{as opposed to being content with an approximation $\widehat{x}_\N(\bmu)$}.  \rev{We point out that this objective is fundamentally different from what is achieved in \cite{ChenGottliebMaday, Elman2015} where preconditioning techniques were developed to drive the reduced solver toward fast convergence to $\widehat{x}_\N(\bmu)$.} We propose two different approaches. In the first approach, we adopt the main ideas of multigrid to devise a reduced basis iteration (RBI) scheme for iteratively solving the system (\ref{eq1}). In the second approach, we use the RBI scheme as a preconditioner for the conjugate gradient (CG) method to accelerate the convergence rate of the CG method.  We present numerical examples to demonstrate the performance of the proposed methods in comparison with multigrid technique. Numerical results show that, \rev{when applied to solve linear systems resulting from discretizing the Poisson's equations, the speed of convergence of our methods matches or surpasses that of the multigrid-preconditioned conjugate gradient method}, while their computational cost per iteration is significantly smaller \rev{providing a feasible alternative when the multigrid approach is out of reach due to timing or memory constraints for large systems}. 
\rev{Moreover, numerical results verify that this new class of reduced basis methods, when applied as a stand-alone solver or as a preconditioner, is capable of achieving the accuracy at the level of the {\em truth approximation} which is far beyond the RB level.} Furthermore, \rev{we propose to incorporate these fast linear solves into the offline procedure of a traditional RBM when $N$ system solves are necessary}. \rev{This adaptive greedy algorithm is capable of alleviating the strenuous offline phase of RBM when the computational cost for snapshot calculation is dominating.}

The paper is organized as follows. In Section 2, we give a brief overview of the reduced basis method. In Section 3, we present the RBI method for solving parametrized SPD linear systems. In Section 4, we introduce the reduced basis conjugate gradient method. In Section 5, we present numerical results to demonstrate their performance. Finally, in Section 6, we end the paper with some concluding remarks.

\section{The reduced basis method}

We briefly review the RB method to compute an approximate solution of the system (\ref{eq1}). For simplicity of exposition, we shall drop the subscript $\N$ in the remainder of this paper.

\subsection{Reduced basis approximation}

We assume that we are given $N$ linearly independent vectors $w_n, 1 \le n \le N,$ of dimension $\N$ to form a RB space $W_{ N} = [w_{1},  \ldots, w_{N}] \in \mathbb{R}^{\N \times N} $. We express the RB approximation as a linear combination of the basis vectors as
\begin{equation}
\label{eq3}
\widehat{x}(\bmu) = W_{ N}  a_N(\bmu) ,
\end{equation}
where $a_N (\bmu)\in \mathbb{R}^N$ is the RB vector. By applying the Galerkin projection of the original system (\ref{eq1}) onto  $W_N$, we find $a_N(\bmu)$ as a solution of the following linear system:
\begin{equation}
\label{eq2}
A_{N}(\bmu) a_N(\bmu) = f_N(\bmu) .
\end{equation}
Here $A_{N}(\bmu)$ is the RB matrix of dimension $N \times N$ and $f_N(\bmu)$ is the RB vector of dimension $N$, which are computed as follows
\begin{equation}
\label{eq2b}
A_{N}(\bmu)  = W_{N}^T A(\bmu) W_{ N}, \qquad  f_N(\bmu) = W_{ N}^T f(\bmu) .
\end{equation}
Note that since $A_{N}(\bmu)$ is symmetric positive-definite, its inverse can be computed by the Cholesky decomposition.

\subsection{Construction of the RB space}

\rev{There are standard greedy algorithms \cite{Rozza08:arcme} that make use of a rigorous (and sometime costly) {\em a posteriori} error estimator, however} we use the greedy algorithm recently proposed in \cite{JiangChenNarayan2018} to construct the RB space $W_N$ as listed in Algorithm \ref{algorithm0}. Instead of  {\em a posteriori} error estimators, this greedy algorithm relies on the $L_1$ norm of the RB vector to choose the next parameter sample from the training set $\Xi_{\rm train}$. It is shown in \cite{JiangChenNarayan2018} that this greedy approach works as effectively as the one using {\rm a posteriori} error estimators \rev{for at least the Poisson-type of systems concerned in this paper. Most importantly, its low cost is appealing in our context as the RB space construction is an overhead expense for the linear system solves}. Note that $W_{ N} = \mbox{MGS}(W_{N-1},  x(\bmu_N))$ denotes the modified Gram-Schmidt orthogonalization  of $x(\bmu_N)$ with respect to the previously selected and orthogonalized basis vectors in $W_{N-1}$.

\begin{algorithm}
\caption{Reduced basis greedy sampling algorithm}
\label{algorithm0}
\vspace{0.5ex}
0. Choose  $\bmu_1$ randomly in $\Xi_{\rm train}$  \\[0.5ex]
1. Initialize $S_{1} = \{\bmu_1\}$ and $W_0 = \emptyset$  \\[0.5ex]
2. \mbox{\textbf{For}} $N = 1,2,\ldots, N_{\max}$  \\[0.5ex]
3. $\quad\ \mbox{Solve }  A(\bmu_N) x(\bmu_N) = f(\bmu_N)$ \\[0.5ex]
4. $\quad\ \mbox{Orthogonalize } W_{ N} = \mbox{MGS}(W_{N-1},  x(\bmu_N))$ \\[0.5ex]
5. $\quad\ \mbox{Solve } A_{N}(\bmu) x_{N} (\bmu) = f_{N} (\bmu) $ for all $\bmu \in \Xi_{\rm train}$ \\[0.5ex]
6. $\quad\ \mbox{Find } \bmu_{N+1} = \arg \max_{\bmu \in \Xi_{\rm train}}  \sum_{n=1}^N |x_{N,n}(\bmu)|$ \\[.5ex]
7. $\quad\ \mbox{Update } S_{N+1} = S_N \cup \bmu_{N+1}$ \\[0.5ex]
8. \mbox{\textbf{End For}}
\vspace{0.5ex}
\end{algorithm}

With the greedy sampling algorithm, we still need to solve the parametrized linear system in Step 3 for $N_{\max}$ times. We shall address how to do these linear solves later. Furthermore, we have to form and solve the RB system in Step 5 for all $\bmu$ in the training set. While solving the RB system is inexpensive, forming it can be much more expensive. However, for a particular set of parametrized linear systems,  the RB system can be formed efficiently through an offline-online procedure as discussed next.

\subsection{Offline-online computational procedure}

For parametrized linear systems in which both $A(\bmu)$ and $f(\bmu)$ are affine in the parameter vector, the corresponding RB matrix $A_N(\bmu)$ and vector $f_N(\bmu)$ can be computed efficiently through an offline-online procedure.  Affine parameter dependence implies that $A(\bmu)$ and $f(\bmu)$  can be expressed as
\begin{equation}
\label{eq12}
A(\bmu) = \sum_{q=1}^{Q} \Theta_q(\bmu) A_q,  \qquad f(\bmu) = \sum_{r=1}^{R} \Phi_r(\bmu) f_r ,
\end{equation}
where $\Theta_q(\bmu), 1 \le q \le Q,$ and $\Phi_r(\bmu), 1 \le r \le R,$ are parameter-dependent functions, while $A_q, 1 \le q \le Q,$ and $f_r, 1\le r \le R,$ are parameter-independent. 

The affine parameter dependence allows us to form $A_N(\bmu)$ and $f_N(\bmu)$ as follows
\begin{equation}
\label{eq12b}
A_{N}(\bmu)  =   \sum_{q=1}^{Q} \Theta_q(\bmu) A_{N, q}, \qquad  f_N(\bmu) = \sum_{r=1}^{R} \Phi_r(\bmu) f_{N,r} ,
\end{equation}
where the following matrices and vectors are pre-computed and stored as
\begin{equation}
\label{eq12c}
A_{N,q} =  W_{N}^T A_q W_{ N},  \qquad  f_{N,r} = W_{ N}^T f_r .
\end{equation}
The computation of $A_{N,q}, 1 \le q \le Q,$ and $f_{N,r}, 1 \le r \le R,$ is performed only once in the offline stage. In the online stage, we form $A_N(\bmu)$ and $f_N(\bmu)$ by (\ref{eq12b}) with $O(QN^2 + RN)$ operations and invert the RB matrix $A_N(\bmu)$ with $O(N^3)$ operations.

In the online stage, we still need to compute the RB approximation $\widehat{x}(\bmu) = W_N a_N(\bmu)$ with $O( \N N)$ operations. Hence, the total operation count of the online stage is $O( \N N + N^3 + Q N^2 + RN)$. Note that $N$, $Q$, and $R$ are typically very small, whereas $\N$ is often very large.  In this case, the computational cost of the RB approximation $\widehat{x}(\bmu)$ is $O( \N N)$ for any $\bmu \in \D$.

\section{The RB iteration method}

In this section, we employ the RB method described earlier to solve the parametrized linear system (\ref{eq1}) for any $\bmu \in \D$. We aim to compute $x(\bmu)$ in an iterative fashion.

\subsection{Main algorithm}

\rev{Inspired by the RB and multigrid methods, we design a scheme which we call Reduced Basis Iteration (RBI) method} for iteratively solving the linear system (\ref{eq1}). 
For any given iterate $x(\bmu)$, we evaluate the residual vector, $r(\bmu) = f(\bmu) - A(\bmu) x(\bmu)$, and  project it onto the RB space to obtain
\begin{equation}
\label{eq5}
r_{N} (\bmu)  = W_{ N}^T r(\bmu).
\end{equation}
We then solve the following system
\begin{equation}
\label{eq6}
A_{N}(\bmu) e_{N} (\bmu) = r_{N} (\bmu) ,
\end{equation}
and update the solution
\begin{equation}
\label{eq7}
x(\bmu) = x(\bmu) +  W_{ N}  e_{N}(\bmu) .
\end{equation}
Finally, we perform the post-smoothing step
\begin{equation}
\label{eq8}
x(\bmu) =  \mbox{S}(A(\bmu), f(\bmu), x(\bmu)) ,
\end{equation}
where $z = \mbox{S}(A,b,y)$ denotes the smoother (Jacobi or Gauss-Seidel) that iterates on $A x  = b$ starting from $y$ and returns $z$. The process is repeated until the maximum number of iterations is reached or the residual norm $\|r(\bmu)\|$ is less than a given tolerance $\epsilon$.  

\begin{algorithm}
\caption{Reduced basis iteration method}
\label{algorithm1}
\vspace{0.5ex}
0. Initialize ${x}(\bmu) = 0$   \\[0.5ex]
1. \mbox{\textbf{For}} $k = 1,2,\ldots, m$  \\[0.5ex]
2. $\quad\ \mbox{Compute the residual } r(\bmu) = f(\bmu) - A(\bmu) x(\bmu) $ \\[0.5ex]
3. $\quad\ \mbox{\textbf{If }} \|r(\bmu)\| < \epsilon$ \textbf{ then} exit loop \\[.5ex]
4. $\quad\ \mbox{Project } r_{N} (\bmu)  = W_{ N}^T r(\bmu)$ \\[0.5ex]
5. $\quad\ \mbox{Solve } A_{N}(\bmu) e_{N} (\bmu) = r_{N} (\bmu) $ \\[0.5ex]
6. $\quad\ \mbox{Update } x(\bmu) = x(\bmu) +  W_{ N}  e_{N}(\bmu) $ \\[.5ex]
7. $\quad\ \mbox{Smooth } x(\bmu) =  \mbox{S}(A(\bmu), f(\bmu), x(\bmu))$ \\[0.5ex]
8. \mbox{\textbf{End For}}
\vspace{0.5ex}
\end{algorithm}

Algorithm \ref{algorithm1} lists the steps of the RBI method. It should be noted that  $x(\bmu) = \widehat{x}(\bmu)$ in Step 6 of Algorithm \ref{algorithm1} for the first iteration. This is true for any initial guess residing in the \rev{column} space of $W_N$. Hence, it is appropriate  to set the initial guess to zero. For later reference, we shall denote by
\begin{equation}
\label{eq8}
y(\bmu) = \mathrm{RBI}(A(\bmu), b(\bmu), W_N, m) ,
\end{equation}
as a procedure that applies the RBI method to the linear system $A(\bmu) x(\bmu) = b(\bmu)$ for $m$ iterations and returns $y(\bmu)$. 

The computational complexity of the RBI method per iteration includes the operation count of a matrix-vector product in Step 2, $O(\N N)$ in both Step 4 and 6, $O(N^{\rev{3}})$ in Step 5, and the operation count of the smoothing in Step 7. Therefore, the computational complexity per iteration is linear in $\N$ if the operation count of both the matrix-vector product and the smoothing is linear in $\N$.

\subsection{Properties}

In what follows, we discuss some basic properties of the RBI method. The first property is the following: 

\begin{lem}
If the solution $x(\bmu)$ of the system (\ref{eq1}) resides in the \rev{column} space of $W_N$, namely, there exists a vector $c(\bmu) = \rev{(c_1(\bmu), c_2(\bmu), \ldots, c_N(\bmu))}\in \mathbb{R}^N$ such that
 \begin{equation}
\label{eq8b}
x(\bmu) = \sum_{j=1}^N c_j(\bmu) w_j  = W_N c(\bmu),
\end{equation}
then $x(\bmu) = \widehat{x}(\bmu)$ and the RBI method converges in one iteration. 
\end{lem}
\begin{proof}
First, we show that $x(\bmu) = \widehat{x}(\bmu)$. By the assumption that the solution of the system (\ref{eq1}) resides in the RB space, we have 
\begin{equation}
A(\bmu) W_N c(\bmu) = f(\bmu) 
\end{equation}
Multiplying both sides of the equation by $W_N^T$ we get
\begin{equation}
A_N(\bmu) c(\bmu) = f_N(\bmu) 
\end{equation}
This implies that $c(\bmu) = a(\bmu)$ since $A_N(\bmu)$ is SPD. Thus, we have $x(\bmu) = \widehat{x}(\bmu)$. Furthermore, since the first iteration of the RBI method immediately yields $x(\bmu) = \widehat{x}(\bmu)$ at Step 6, the method converges in one iteration. 
\end{proof}

Since the method will converge in one iteration for any $x(\bmu)$ belonging to the column space of $W_N$, we expect that increasing $N$ will render the method converge faster. This is because increasing $N$ will effectively make the column space of $W_N$ closer to the solution $x(\bmu)$ and thus accelerate the convergence rate.

\begin{lem}
In the absence of the smoother $\mathrm{S}()$, the RBI method stagnates  at $x(\bmu) = \widehat{x}(\bmu)$.
\end{lem}
\begin{proof}
In the absence of the smoother, the first iteration of the RBI method yields $x(\bmu) = \widehat{x}(\bmu)$. The second iteration of the method yields
 \begin{equation}
\label{eq8c}
x(\bmu) = \widehat{x}(\bmu) +  \widehat{x}(\bmu) - W_N A^{-1}_N(\bmu) W_N^T  A(\bmu) \widehat{x}(\bmu) .
\end{equation}
Since $\widehat{x}(\bmu) = W_N a(\bmu)$ we have $ \widehat{x}(\bmu) - W_N A^{-1}_N(\bmu) W_N^T  A(\bmu) \widehat{x}(\bmu) = 0$. It means that $x(\bmu) = 
\widehat{x}(\bmu)$ in the second iteration. As a consequence, the RBI method stagnates at $x(\bmu) = \widehat{x}(\bmu)$.
\end{proof}

Therefore, the smoother plays a crucial role in ensuring that the RBI method does not stagnate. The smoother lifts the current iterate out of the \rev{column} space of $W_N$, thereby avoiding stagnation. Furthermore, the smoother must ensure that its output converges toward the solution. For symmetric positive-definite systems, \rev{Gauss-Seidel method} is known as one of the best smoothers for multigrid methods because it can remove high-frequency features on the fine mesh very effectively, so that the Galerkin projection can be resolved on coarser meshes. 

\begin{lem}
Let us introduce the error $e(\bmu)$ that satisfies the following error equation 
\begin{equation}
A(\bmu) e(\bmu) =  r(\bmu) .
\end{equation}
Then $\widehat{e}(\bmu) = W_N e_N(\bmu)$ is the RB approximation to $e(\bmu)$.
\end{lem}
\begin{proof}
Let $\tilde{e}(\bmu) = W_N g_N(\bmu)$ be the RB approximation to $e(\bmu)$. Then the vector $g_N(\bmu)$ can be found as the solution of the following system:
\begin{equation}
W_N^T A(\bmu) W_N g_N(\bmu) =  W^T_N r(\bmu) .
\end{equation}
It thus follows that $g_N(\bmu) = e_N(\bmu)$, where $e_{N}(\bmu)$ is given by Step 5 of the RBI algorithm. As a result, we have $\widehat{e}(\bmu) = \tilde{e}(\bmu)$. This concludes the proof.
\end{proof}

In Step 6 of the RBI algorithm, we add the RB approximation of the error to the current iterate. Therefore, the convergence of the RBI method depends on how well $\widehat{e}(\bmu)$ approximates $e(\bmu)$. The performance of the RBI method depends on two factors: ({\em i}) the approximability of the RB method with respect to the solution of the linear system (\ref{eq1}) and ({\em ii}) the convergence \rev{rate} of the smoother.

\subsection{Relation to multigrid method}

The RBI scheme is similar to a {\em two-level} multigrid method. The main difference between our scheme and the two-level multigrid method lies in the restriction and interpolation operators. In our approach, we use the RB matrix $W_N$ as the interpolation operator and its transpose as the restriction operator. For the two-level multigrid method, the restriction operator transfers vectors from the fine grid of dimension $\N$ to the coarse grid of dimension $\N/2^d$, while the interpolation operator transfers vectors from the coarse grid back to the fine grid. (Note here that $d$ is the dimension of the physical domain in which the solution vector $x_\N(\bmu)$ results from the numerical discretization of a physical problem.) As a result, when $\N$ is very large, it is necessary for multigrid method to use a sequence of grids, such that the linear systems on the coarsest grid can be solved very fast. \rev{In contrast, because the RB system can be solved very fast, it is not necessary for the RBI scheme to use more than two levels.} Due to the good approximation property of the subspace adapting to the low-dimensional manifold of the parametrized system (\ref{eq1}), the RBI scheme can converge faster than the multigrid method. 
\rev{Moreover, with $W_N$ being an assembly of solutions to the linear system at various parameter values, our ``restriction'' and ``interpolation'' operators are easily created without resorting to any structure of the underlying mesh or the degrees of freedom. This unique feature renders our method applicable to scenarios when (geometric) multigrid method finds challenging.}

\rev{Furthermore, the smoother plays a different role in the RBI scheme and a multigrid solver. In the context of a multigrid solver, the main role of the smoother is to smooth the error rather than reduce it. For the RBI scheme, the main role of the smoother is to reduce the error rather than smooth it. For symmetric positive-definite systems, Gauss-Seidel method can smooth and reduce the error. For this reason and for comparison with multigrid method, we consider Gauss-Seidel method as the smoother for all the methods described in this paper.}

\section{The reduced-basis conjugate gradient method}

\subsection{Main algorithm}

\rev{In this section, }we employ the RBI scheme as a preconditioner in the conjugate gradient method for iteratively solving the parametrized linear system (\ref{eq1}). The reduced basis conjugate gradient (RBCG) method is listed in Algorithm \ref{algorithm2}. For later reference, we shall denote by
\begin{equation}
\label{eq8rt}
y(\bmu) = \mathrm{RBCG}(A(\bmu), b(\bmu), W_N) ,
\end{equation}
as a solution procedure that applies the RBCG method to the linear system $A(\bmu) x(\bmu) = b(\bmu)$ and returns $y(\bmu)$ as the solution. 

\begin{algorithm}
\caption{Reduced basis conjugate gradient method}
\label{algorithm2}
\vspace{0.5ex}
0. Start with $x_0(\bmu) = 0$ and set $k = 0$ \\[0.5ex]
1. $r_k(\bmu) = f(\bmu) - A \, x_k(\bmu)$   \\[0.5ex]
2. $y_k(\bmu) = \mathrm{RBI}(A(\bmu), r_k(\bmu), W_N, 1)$   \\[0.5ex]
3. $p_k(\bmu) = y_k(\bmu)$   \\[0.5ex]
4. \mbox{\textbf{Repeat}} \\[0.5ex]
5. $\quad\ \alpha_k(\bmu) = \displaystyle \frac{r_k^T(\bmu) y_k(\bmu)}{p_k^T(\bmu) A(\bmu) p_k(\bmu)} $ \\[0.5ex]
6. $\quad\ x_{k+1}(\bmu) = x_k(\bmu) + \alpha_k(\bmu) p_k(\bmu) $ \\[0.5ex]
7. $\quad\ r_{k+1}(\bmu) = r_k(\bmu) - \alpha_k(\bmu) A(\bmu) p_k(\bmu)$ \\[0.5ex]
8. $\quad\ \mbox{\textbf{If }} \|r_{k+1}(\bmu)\| < \epsilon$ \textbf{ then} exit loop \\[.5ex]
9. $\quad\ y_{k+1}(\bmu) = \mathrm{RBI}(A(\bmu), r_{k+1}(\bmu), W_N, 1)$ \\[0.5ex]
10. $\quad \beta_{k}(\bmu) =  \displaystyle \frac{y_{k+1}^T(\bmu) r_{k+1}(\bmu)}{y_k^T(\bmu) r_k(\bmu)} $ \\[0.5ex]
11. $\quad p_{k+1}(\bmu) = y_{k+1}(\bmu) + \beta_k(\bmu) p_{k}(\bmu)$ \\[0.5ex]
12. $\quad k = k+1$ \\[0.5ex]
13. \mbox{\textbf{End Repeat}}
\vspace{0.5ex}
\end{algorithm}

It is important to note that we  perform only one iteration of the RBI preconditioner. As a result, there is no matrix-vector multiplication required for the RBI preconditioner. Therefore, the number of matrix-vector multiplications is exactly equal to the number of iterations. The operation count of a matrix-vector multiplication is equal to the number of non-zeroes of the matrix. If the matrix $A(\bmu)$ is dense then the operation count of one matrix-vector multiplication is $O(\N^2)$. On the other hand, if the matrix $A(\bmu)$ is sparse then the operation count of one matrix-vector multiplication is $O(\N M)$, where $M \ll \N$ represents the sparsity level of the matrix. As a result, for sparse linear systems, the total operation count of the RBCG method per iteration is thus $O(\N (N+M))$, assuming that the cost of the smoothing procedure is smaller than $O(\N (N+M))$. 

We see that the RBCG method has the same computational complexity per iteration as the RBI method. It is expected that the RBCG method will converge faster than the RBI method owing to the optimal Krylov subspace of the CG method for SPD linear systems. 

An important question is how do we determine $N$ for the RBCG method. Since $W_N \subset W_{N+1}$ for any $N$, we expect that increasing $N$ makes the RBCG method converge faster at the expense of higher computational cost per iteration.  To address this issue, we propose to change $N$ adaptively. When the convergence of the RBCG method is slow, namely, when the residual norm decays slowly during the iteration, we keep increasing $N = N + 1$ until the convergence is fast enough.  For instance, if the residual norm ratio $\|r_k(\bmu)\|/\|r_{k+1}(\bmu)\| < \gamma$ (says $\gamma = 10$), then we increase $N$ by 1. In this way, we enforce the residual norm to drop at least by a factor of $\gamma$.

\subsection{Adaptive greedy sampling}
\rev{Here, we propose to incorporate our RBCG algorithm back into the standard RB method. In particular, during the offline phase of RB when the greedy algorithm is invoked to identify the $N+1^{\rm th}$ basis, we have $N$ bases and can readily form $W_N$ to speed up the linear system solve for $x(\bmu_{N+1})$. The resulting adaptive greedy sampling scheme is given by Algorithm \ref{algorithm3}. Note that for $N=1$, a standard linear solver such as multigrid method is used to solve the linear system $A(\bmu_1) x(\bmu_1) = f(\bmu_1)$.}

\begin{algorithm}
\caption{Adaptive greedy sampling}
\label{algorithm3}
\vspace{0.5ex}
0. Choose  $\bmu_1$ randomly in $\Xi_{\rm train}$  \\[0.5ex]
1. Initialize $S_{1} = \{\bmu_1\}$ and $W_0 = \emptyset$   \\[0.5ex]
2. \mbox{\textbf{For}} $N = 1,\ldots, N_{\max}$  \\[0.5ex]
3. $\quad\ \mbox{Solve }  x(\bmu_N) = \mbox{RBCG} (A(\bmu_N), f(\bmu_N), W_{N-1})$ \\[0.5ex]
4. $\quad\ \mbox{Orthogonalize } W_{ N} = \mbox{MGS}(W_{N-1},  x(\bmu_N))$ \\[0.5ex]
5. $\quad\ \mbox{Solve } A_{N}(\bmu) x_{N} (\bmu) = f_{N} (\bmu) $ for all $\bmu \in \Xi_{\rm train}$ \\[0.5ex]
6. $\quad\ \mbox{Find } \bmu_{N+1} = \arg \max_{\bmu \in \Xi_{\rm train}} \sum_{n=1}^N |x_{N,n}(\bmu)|$ \\[.5ex]
7. $\quad\ \mbox{Update } S_{N+1} = S_N \cup \bmu_{N+1}$ \\[0.5ex]
8. \mbox{\textbf{End For}}
\vspace{0.5ex}
\end{algorithm}

\rev{Furthermore, in the context of numerical approximation of an elliptic PDE, we can further reduce the computational cost of the offline stage by pursuing a a multi-fidelity approach for the construction of $W_N$. In particular, we run the adaptive greedy sampling algorithm using a {\em coarse grid} to obtain the parameter sample set $S_{N}$. After computing $S_N$ on the coarse grid, we then generate the matrix $W_N$ by solving the PDE $N$ times on a {\em fine grid}. The facts that we adopt this multi-fidelity approach and that the RB dimension can be rather low for the purpose of this paper make the overhead time negligible. This is confirmed by the computation time taking into account the offline cost, as reported in the next section.}

\section{Numerical experiments}

\label{sec:results}

In this section, we present numerical examples to demonstrate the efficiency of the proposed iterative solvers and preconditioners.  \rev{Since this paper focuses on symmetric positive definite systems, we test the methods on linear systems resulting from the discretization of the Possion equation by Finite Element Method \cite{brenner_mathematical_2007} as implemented in the Matlab package iFEM \cite{Chen2009iFEM}}. Indeed, we solve the following equation on $\Omega = [0,1]^3 \subset {\mathbb R}^3$ by the piecewise linear continuous Galerkin method. 
\begin{subequations}
\begin{equation}
\label{eq:generalpoisson}
\nabla \cdot \left( \kappa({\bf x}; \bmu) \nabla u\right) = 3 \pi^2 \sin(\pi x) \sin(\pi y) \sin(\pi z)  \quad \forall \bm{x} \in \Omega, \bm{\mu} \in \mathcal{D}, 
\end{equation} 
\begin{equation}
\label{eq:bdrycondition}
u(x) = g_D(x; \bmu). 
\end{equation}
\end{subequations}
The coefficient $\kappa({\bf x}; \bmu)$ and the boundary condition $g_D({\bf x}; \bmu)$ are taken as in Table \ref{tab:setup} for our two test cases. \rev{The FE discretization of the above Poisson equation results in the following linear system:
\begin{equation}
A_\N(\bmu) u_\N(\bmu) =  b_\N(\bmu) ,
\end{equation} 
where $u_\N(\bmu)$ denotes the vector of degrees of freedom of the FE solution $u_h(\bmu)$. In both cases, the stiffness matrix is expressed as $A_\N(\bmu) =  A_{1,\N} + \mu_1 A_{2,\N}$. The load vector  is independent of $\bmu$ in the first case and has the parameter-dependent form $b_\N(\bmu) = (1-\mu_2) b_{1,\N} + \mu_2 b_{2,\N}$ in the second case.} The \rev{finite element} solutions for representative parameter values are depicted in Figure \ref{fig:soln}.
\begin{figure}
\begin{center}
\includegraphics[width=0.99\textwidth]{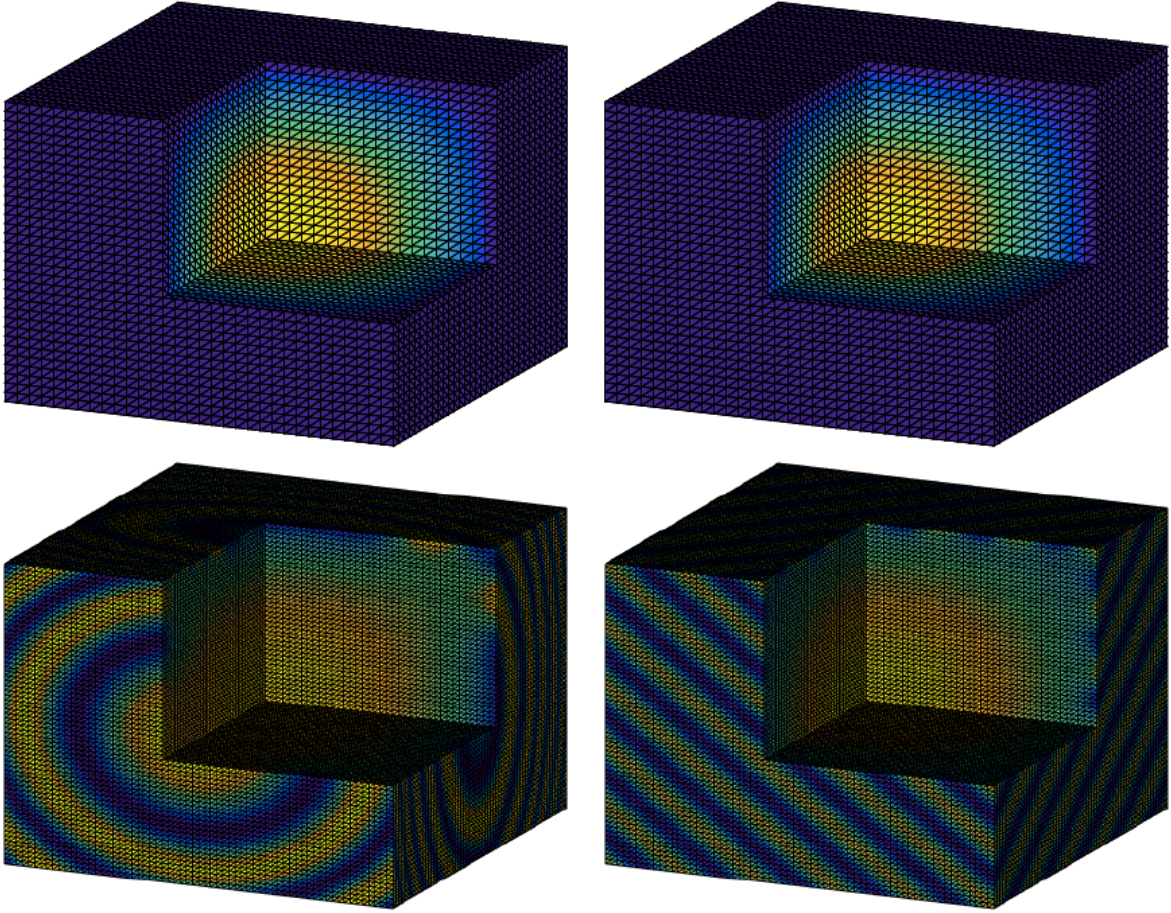}
\end{center}
\caption{\rev{Representative finite element} solutions for the two cases. On the top are for $\bmu = 0$, and $2$. At the bottom are for Case 2 with ${\bf \bmu} = (0, 0)$, and $(2,1)$. }
\label{fig:soln}
\end{figure}

\begin{table}
\begin{center}
    \renewcommand{\tabcolsep}{0.2cm}
    \renewcommand{\arraystretch}{1.6}
\begin{tabular}{|c|c|c|}
\hline
 & Case 1 & Case 2 \\
\hline
$\kappa({\bf x}; \bmu)$ 
& 
\begin{minipage}[t]{0.3\textwidth}
\begin{center}
$\kappa_1({\bf x}; \bmu) = 1 + \mu \left( (x - \frac{1}{2})^2 \right.$\\ 
\hspace*{7mm} $\left.+ (y - \frac{1}{2})^2 + (z - \frac{1}{2})^2 \right)$
\end{center}
\end{minipage}
& 
\begin{minipage}[t]{0.5\textwidth}
\begin{center}
$\kappa_2({\bf x}; \bmu) = 1 + \mu_1 \left(\sin\left(20 \pi (4 (x - \frac{1}{2})^2\right. \right.$ \\
\hspace*{7mm}$\left.\left.+ (y - \frac{1}{2})^2 + (z - \frac{1}{2})^2\right) \right)^2$
\end{center}
\end{minipage}
\\
\hline
$g_D({\bf x}; \bmu)$ & $g_{D1}({\bf x}; \bmu) = 0$ & 
\begin{minipage}[t]{0.4\textwidth}
\begin{center}
$g_{D2}({\bf x}; \bmu) = 
(1 - \mu_2) \cos \left (10 \pi \left (4(x -  \frac{1}{2})^2 \right.\right.$\\
$\left. \left.+ (y -  \frac{1}{2})^2 + (z -  \frac{1}{2})^2\right) \right)$ \\
$+\mu_2 \cos(10 \pi (x+y+z))$
\end{center}
\end{minipage}
\\
\hline
\end{tabular}
\end{center}
\caption{Set up of the two parametric systems.}
\label{tab:setup}
\end{table}

\subsection{Convergence of the new schemes}
\begin{figure}
\begin{center}
\includegraphics[width=0.49\textwidth]{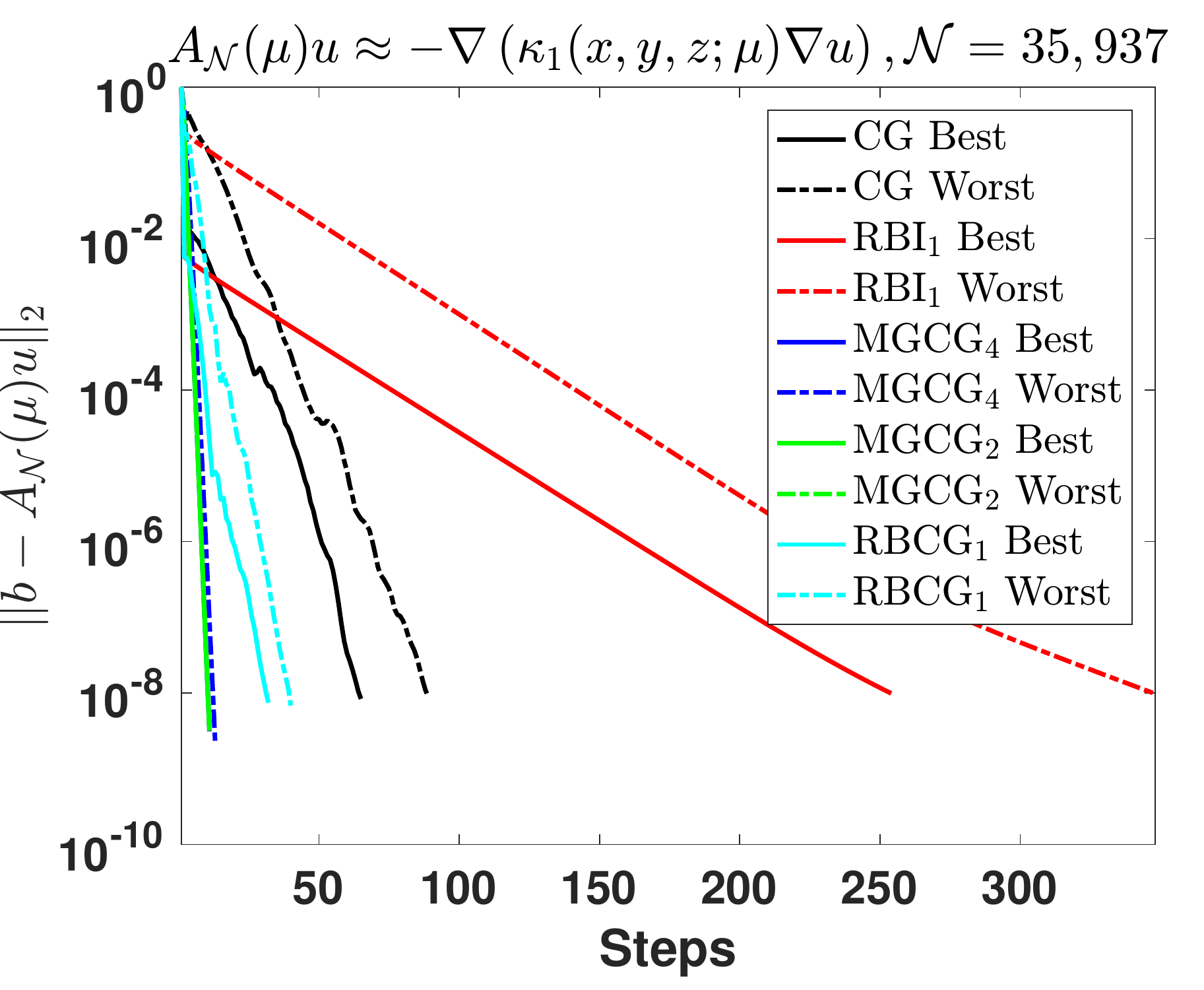}
\includegraphics[width=0.49\textwidth]{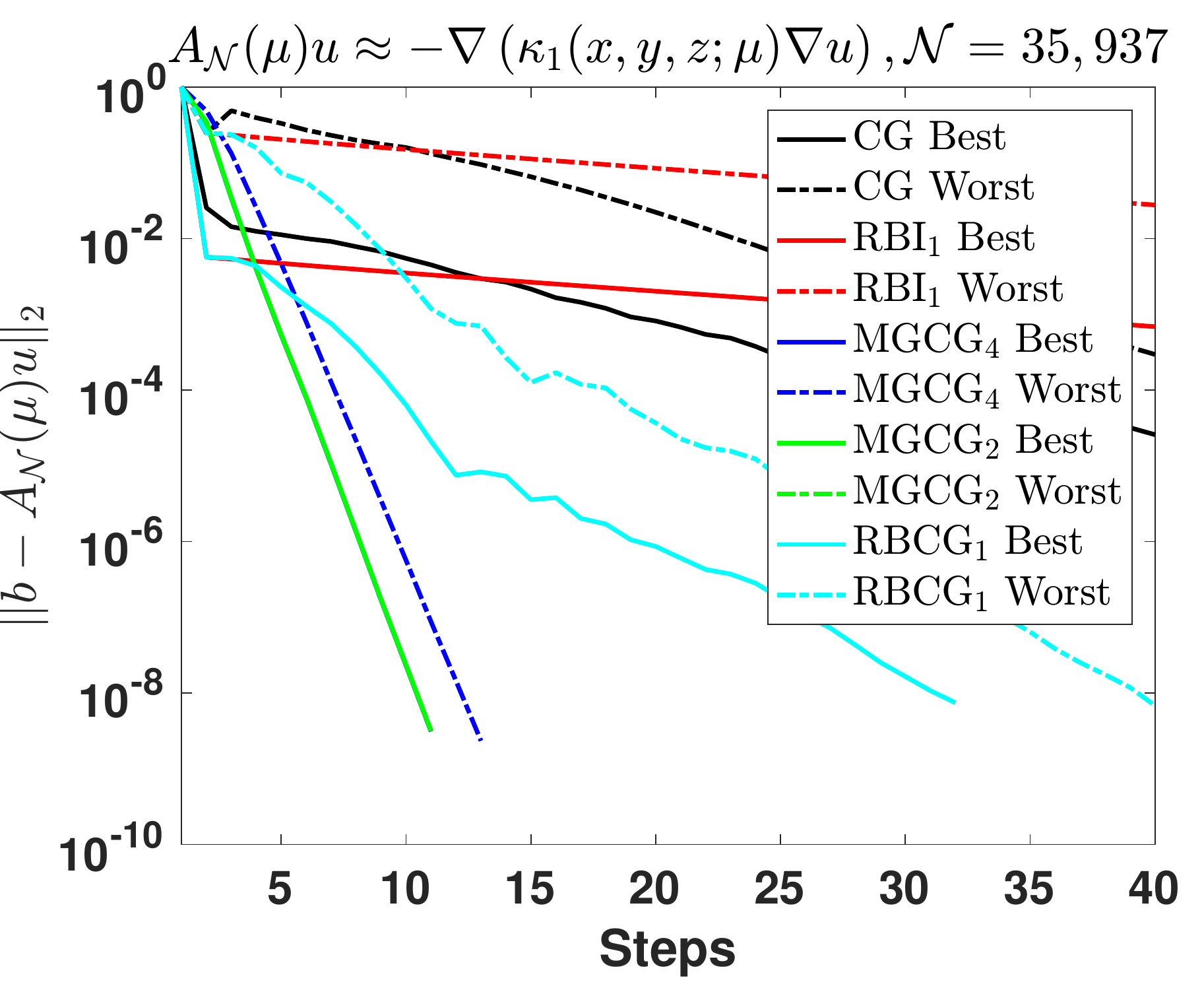}
\includegraphics[width=0.49\textwidth]{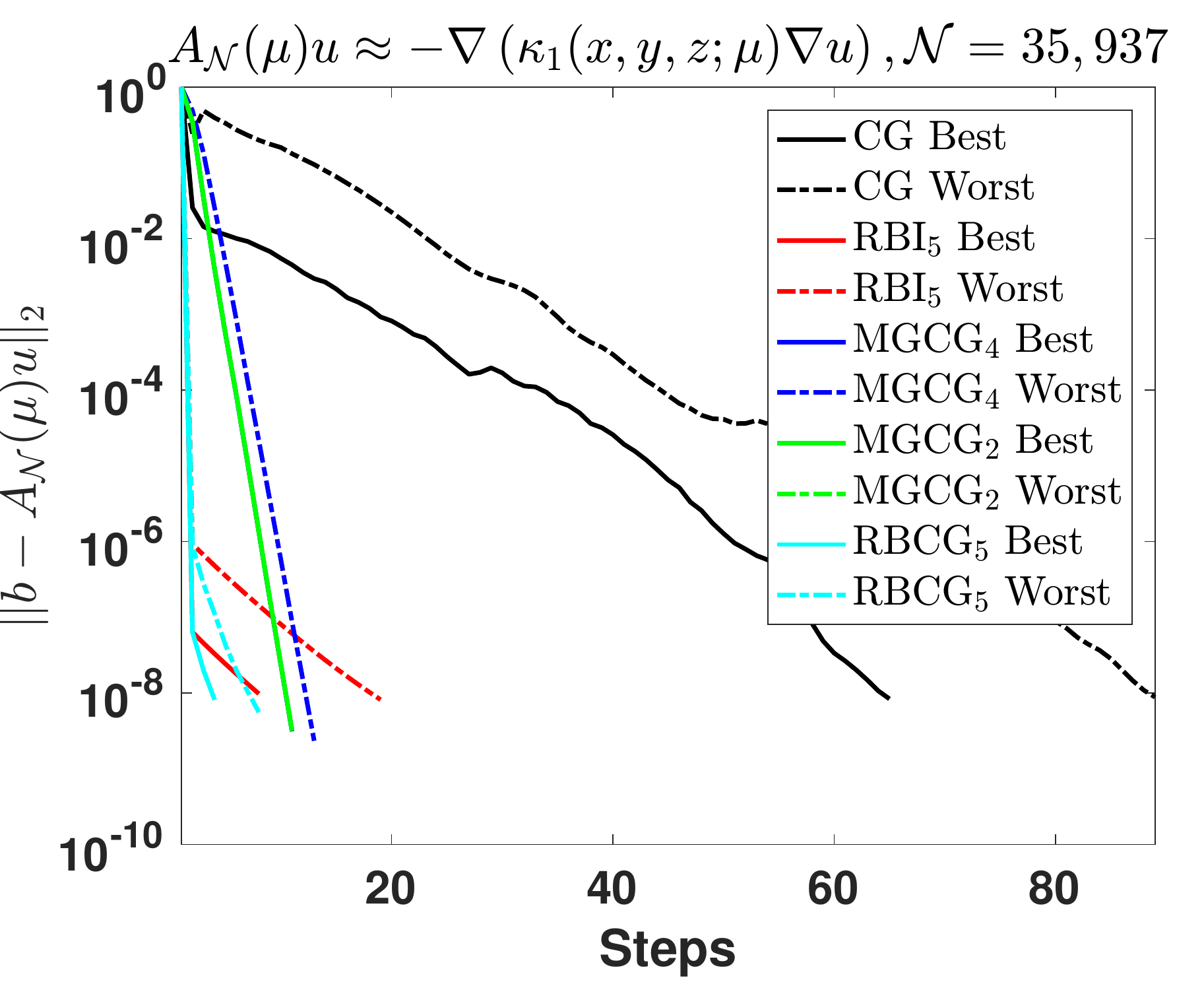}
\includegraphics[width=0.49\textwidth]{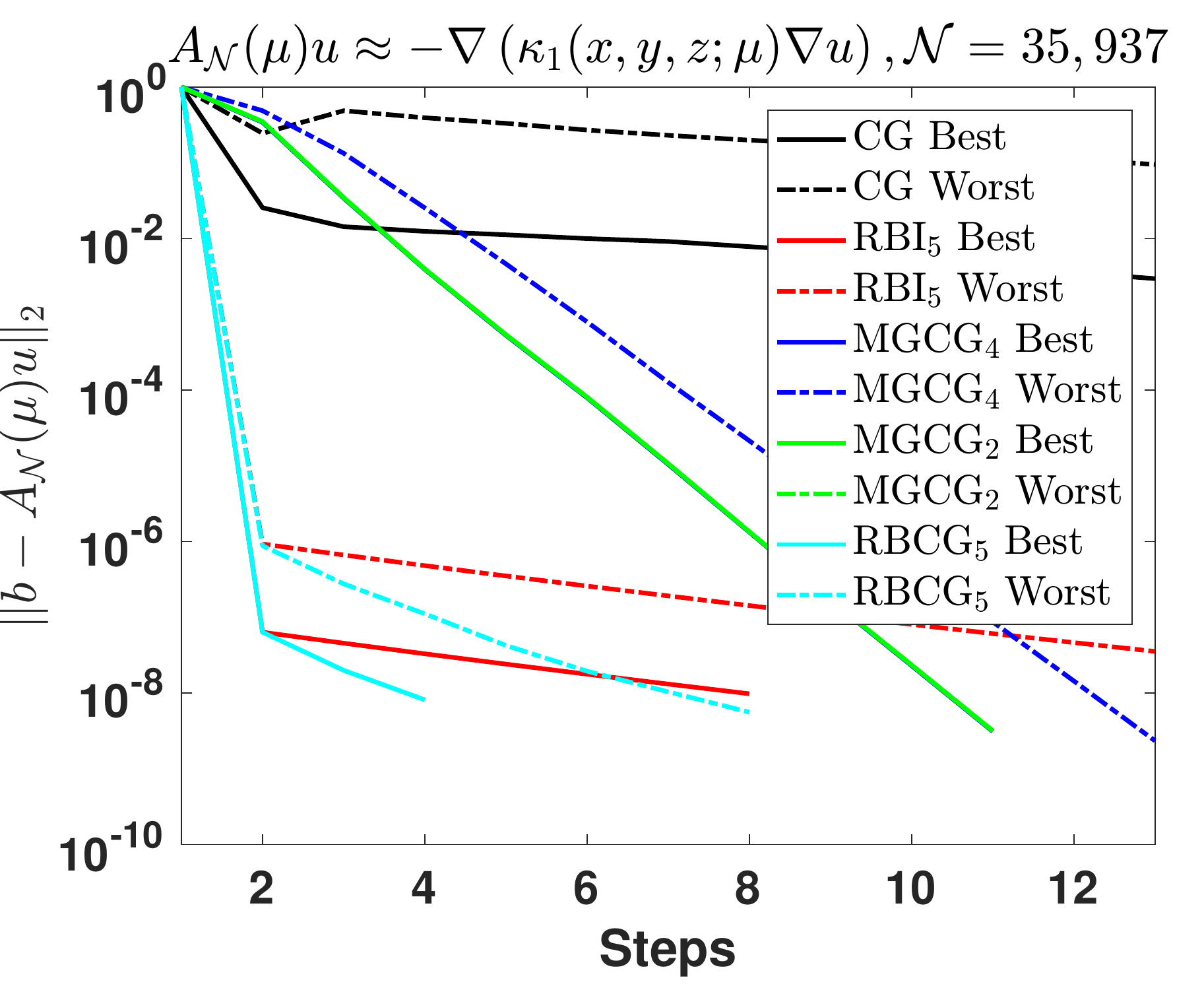}
\end{center}
\caption{The histories of convergence of the relative error for Case 1, as defined by \eqref{eq:relerr}, for each scheme considered with RB dimension being 1 for top and 5 for bottom. The right is the zoomed in version of the left.}
\label{fig:conv_case1}
\end{figure}

\begin{figure}
\begin{center}
\includegraphics[width=0.49\textwidth]{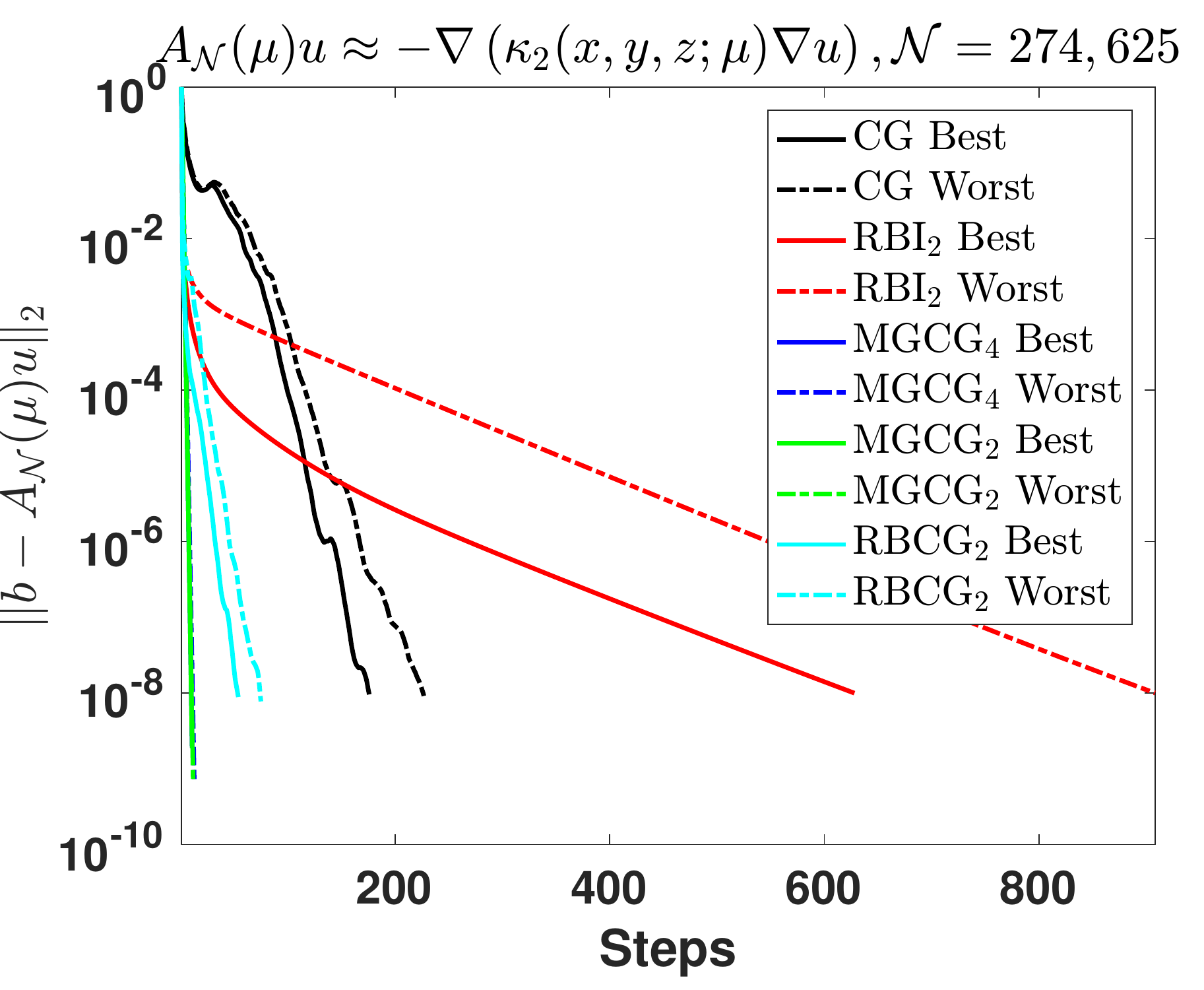}
\includegraphics[width=0.49\textwidth]{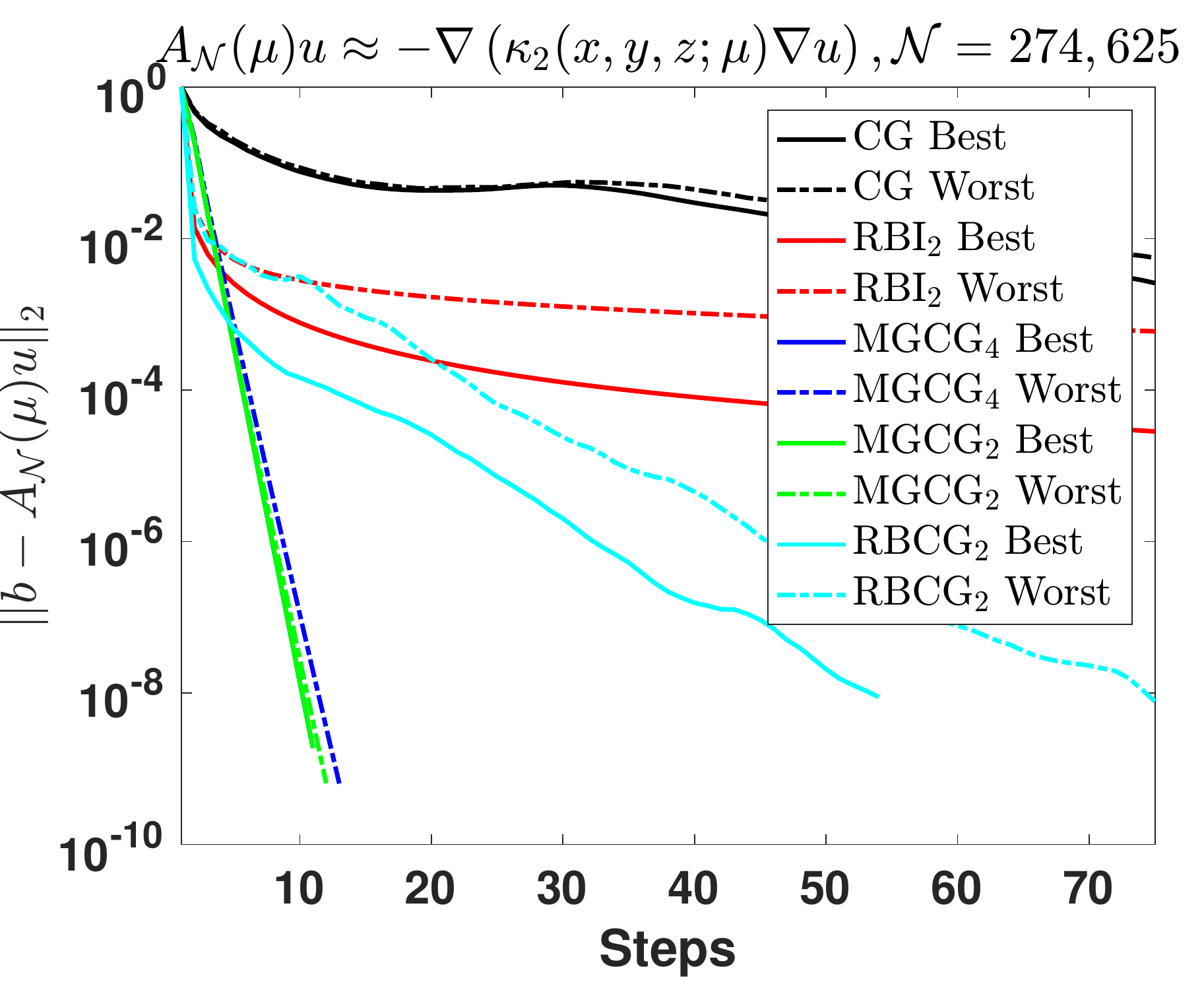}
\includegraphics[width=0.49\textwidth]{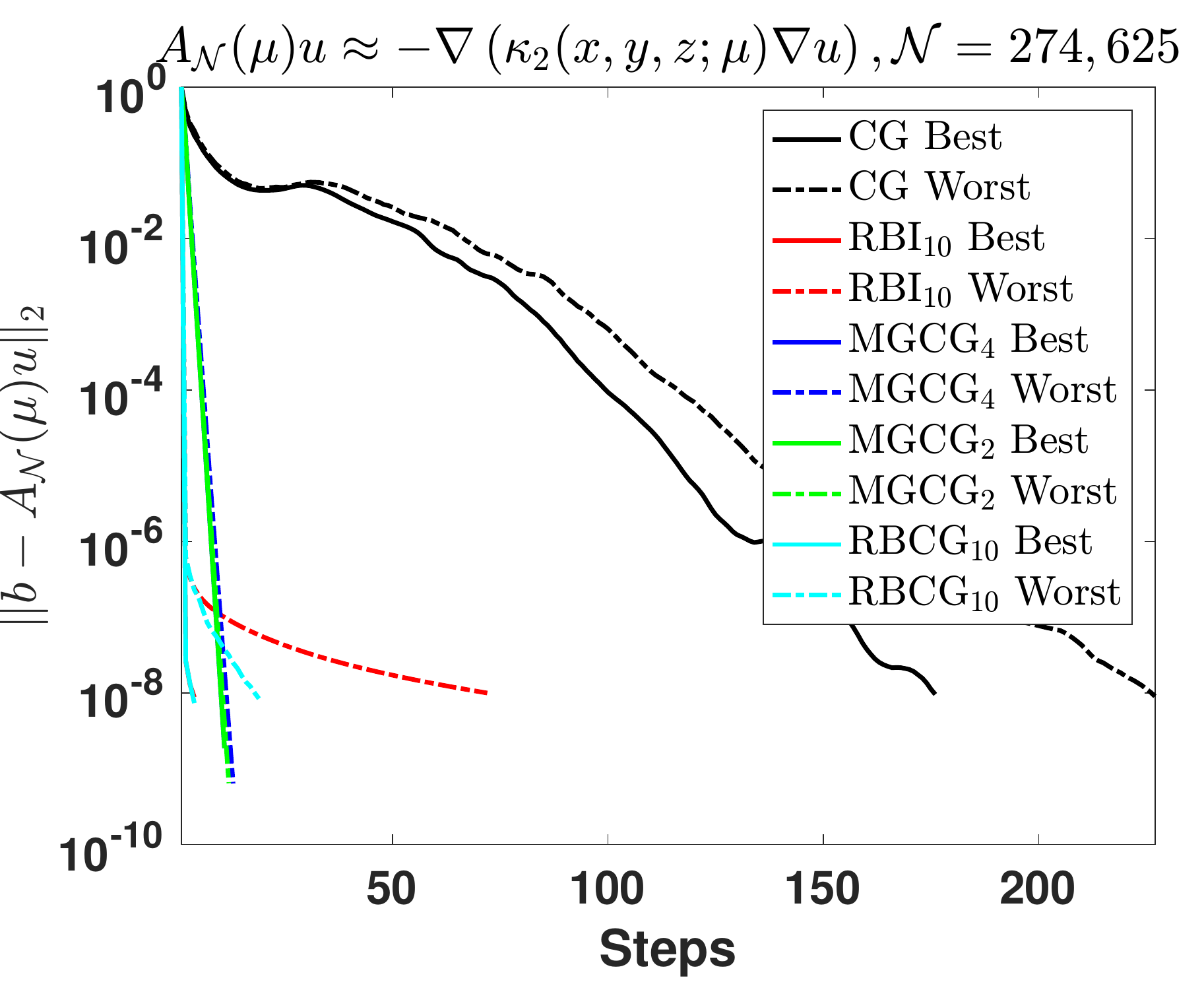}
\includegraphics[width=0.49\textwidth]{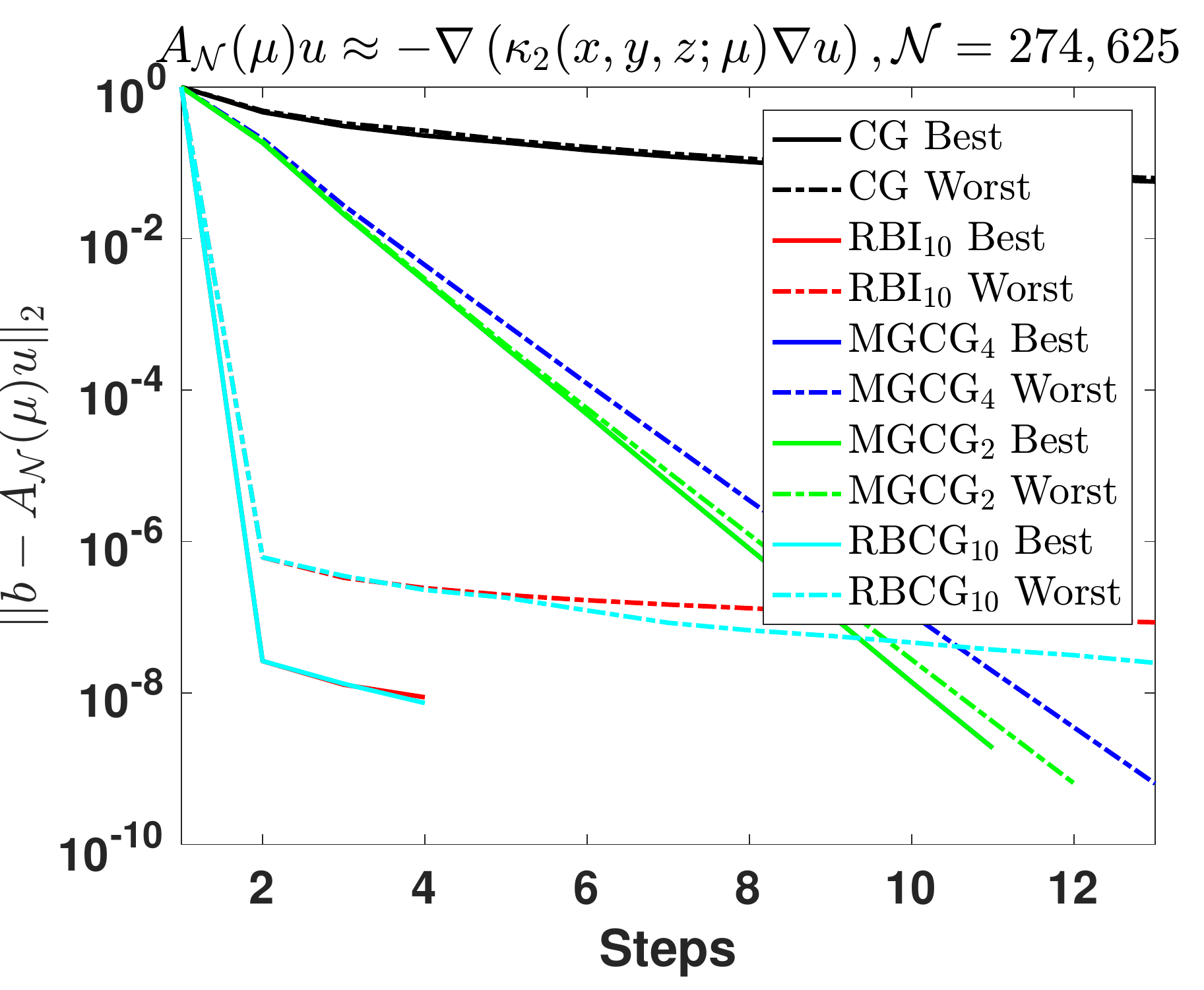}
\end{center}
\caption{The histories of convergence of the relative error for Case 2, as defined by \eqref{eq:relerr}, for each scheme considered with RB dimension being 2 for top and 20 for bottom. The right is the zoomed in version of the left.}
\label{fig:conv_case2}
\end{figure}

\rev{First, we look at the convergence rate of both the RBI scheme and the RBCG method, and compare it with that of the conjugate gradient method with and without multigrid as a preconditioner.} The histories of convergence for each scheme are reported in Figure \ref{fig:conv_case1} and Figure \ref{fig:conv_case2}. Here, we test the methods at $100$ randomly chosen parameter values in their domain $[0, 1]$ and $[0, 2] \times [0, 1]$ respectively for the two cases, and observe the relative error of residual in $L^2$ norm:
\begin{equation}
\epsilon(x; \bmu) = \frac{\lVert b(\bmu) - A(\bmu) u(\bmu) \rVert}{ \lVert b(\bmu) \rVert}.
\label{eq:relerr}
\end{equation}
For each scheme, we plot the best and worst case scenario, i.e. the cases that take the largest and smallest number of steps for convergence among all 100 \rev{parameter values}.

For Case 1, as shown in Figure \ref{fig:conv_case1}, we report two scenarios, $N = 1$ and $N = 5$. We clearly see that, even if we only use a one-dimensional RB as a preconditioner (top row of Figure \ref{fig:conv_case1}, the method converges in half number of steps. When we raise the dimension of the RB preconditioner to 5, this ratio decreases to $\frac{1}{10}$. This demonstrates the high level of effectivity of the new scheme. For Case 2, we test the schemes for $N = 2$ and $10$. It also demonstrate that the RBI-based schemes vastly outperforms the CG method, while outperforming or matching the multigrid-preconditioned CG methods when it comes to the number of steps toward convergence.

\subsection{Computation time}

\rev{Next, we examine the computation time for each scheme}. We show in Figure \ref{fig:time} the cumulative computation time for each scheme as we increase the RB dimension. Of course, the CG and multigrid-preconditioned CG stay the same. For that reason, we use the CG timing as reference. We see that, as the RB dimension increases, the RB schemes becomes more economical going from being more expensive than CG to much less expensive. Comparing with the multigrid-preconditioned CG, we clearly see that, even though the number of steps are only slightly better or comparable as shown in Figures \ref{fig:conv_case1} and \ref{fig:conv_case2}), the RB-based schemes are orders of magnitude more efficient. This of course is because each iteration of the RB schemes is much fast than one iteration of the multigrid-preconditioned CG.

\begin{figure}
\begin{center}
\includegraphics[width=0.49\textwidth]{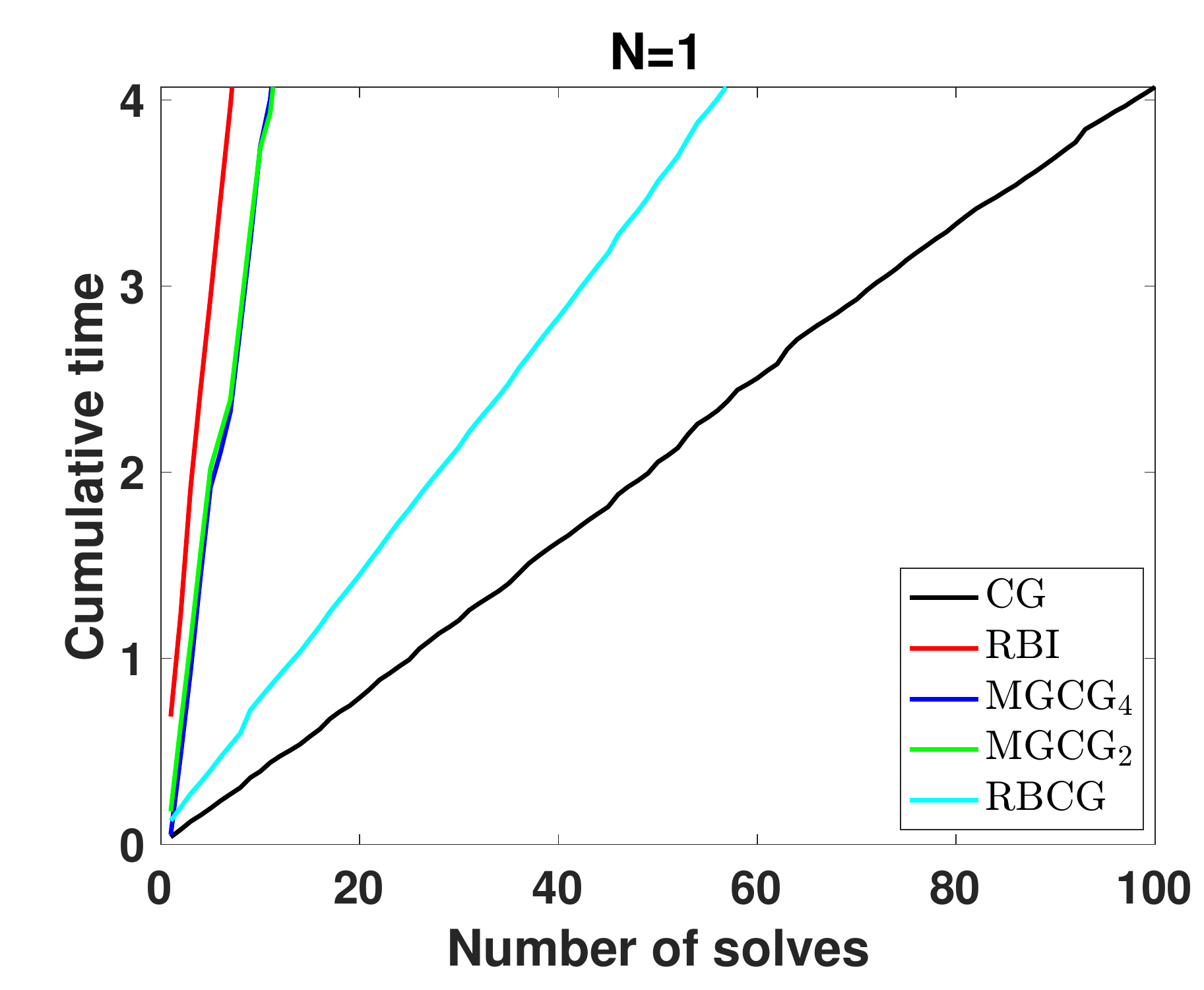}
\includegraphics[width=0.49\textwidth]{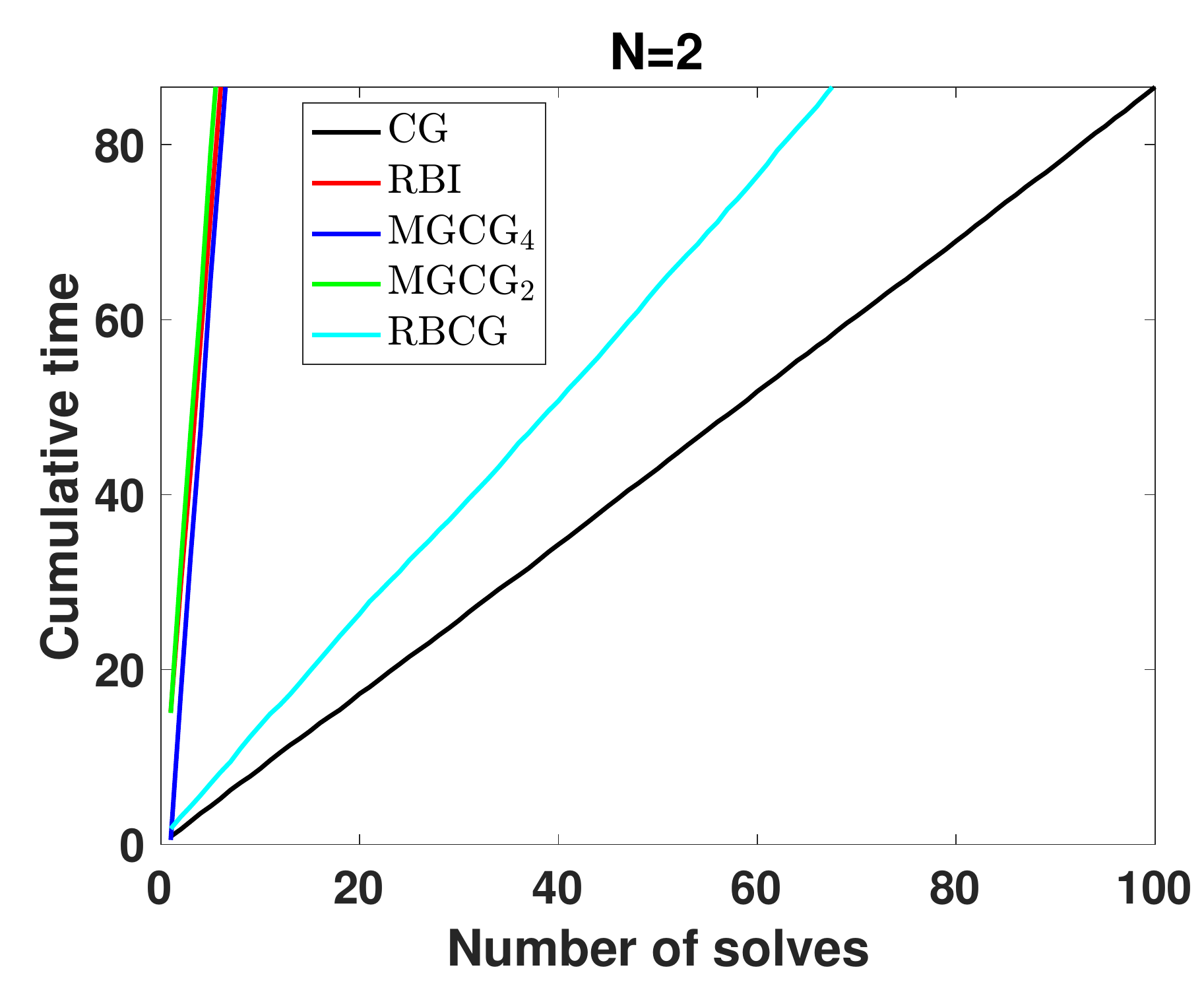}\\
\includegraphics[width=0.49\textwidth]{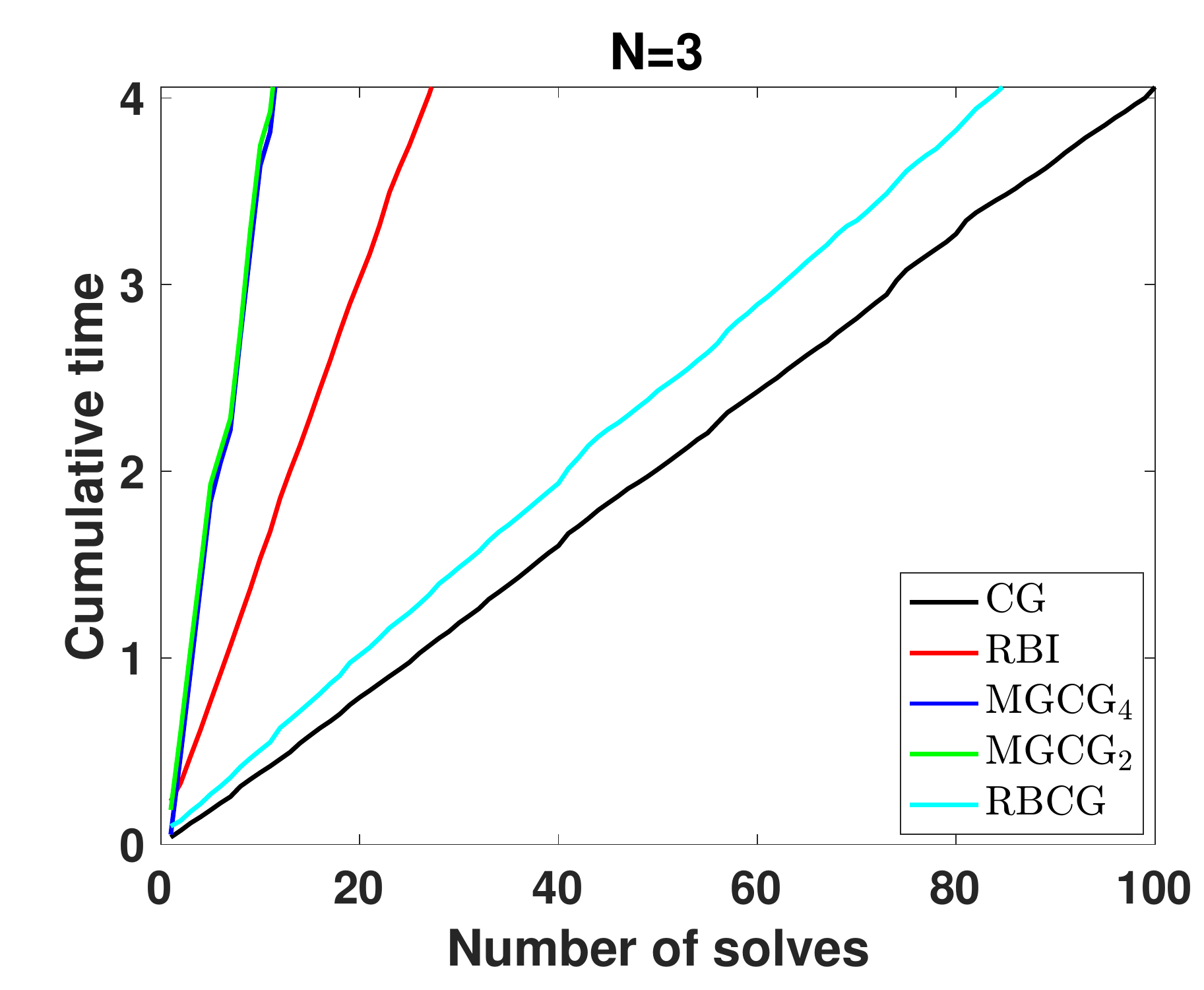}
\includegraphics[width=0.49\textwidth]{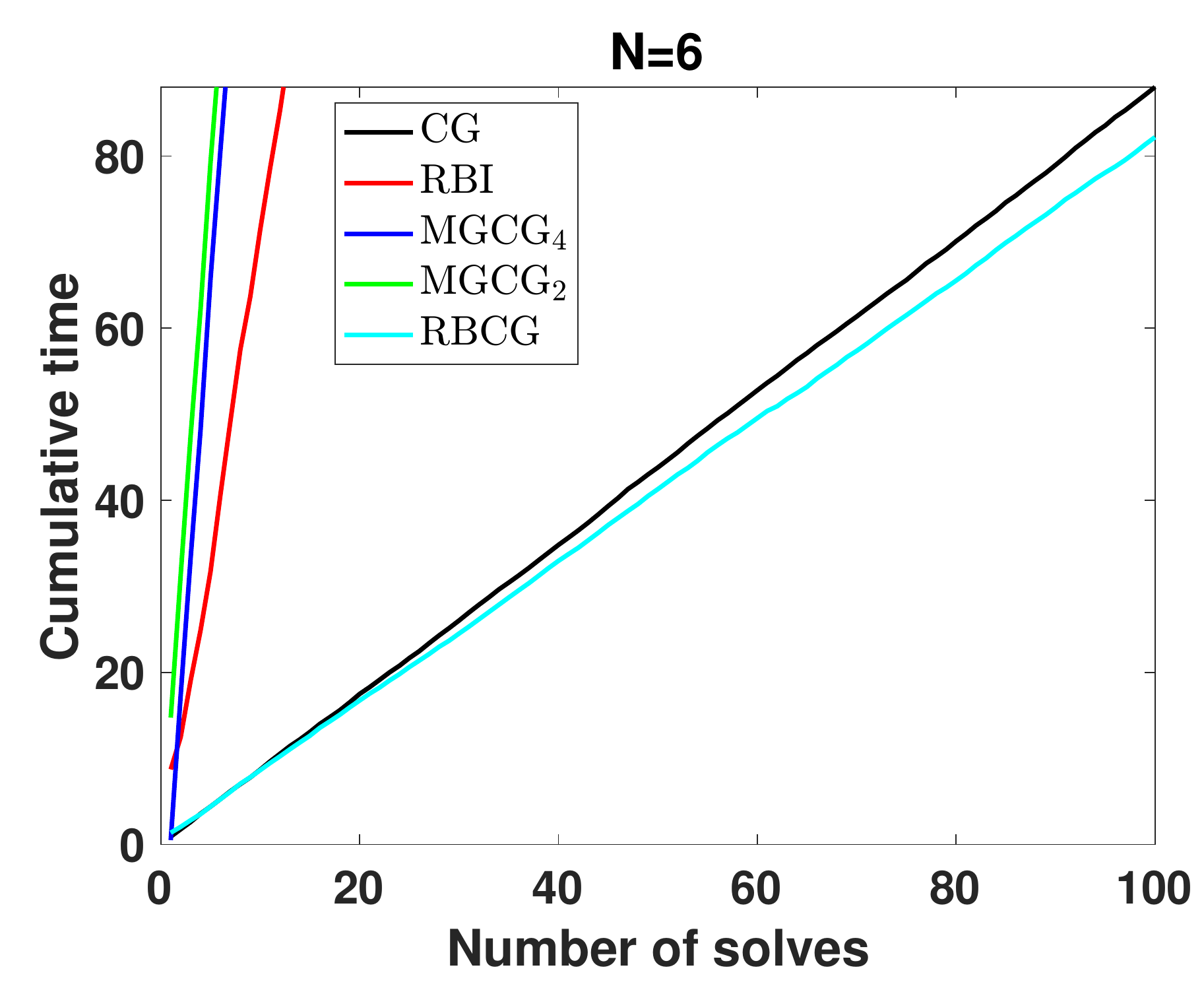}\\
\includegraphics[width=0.49\textwidth]{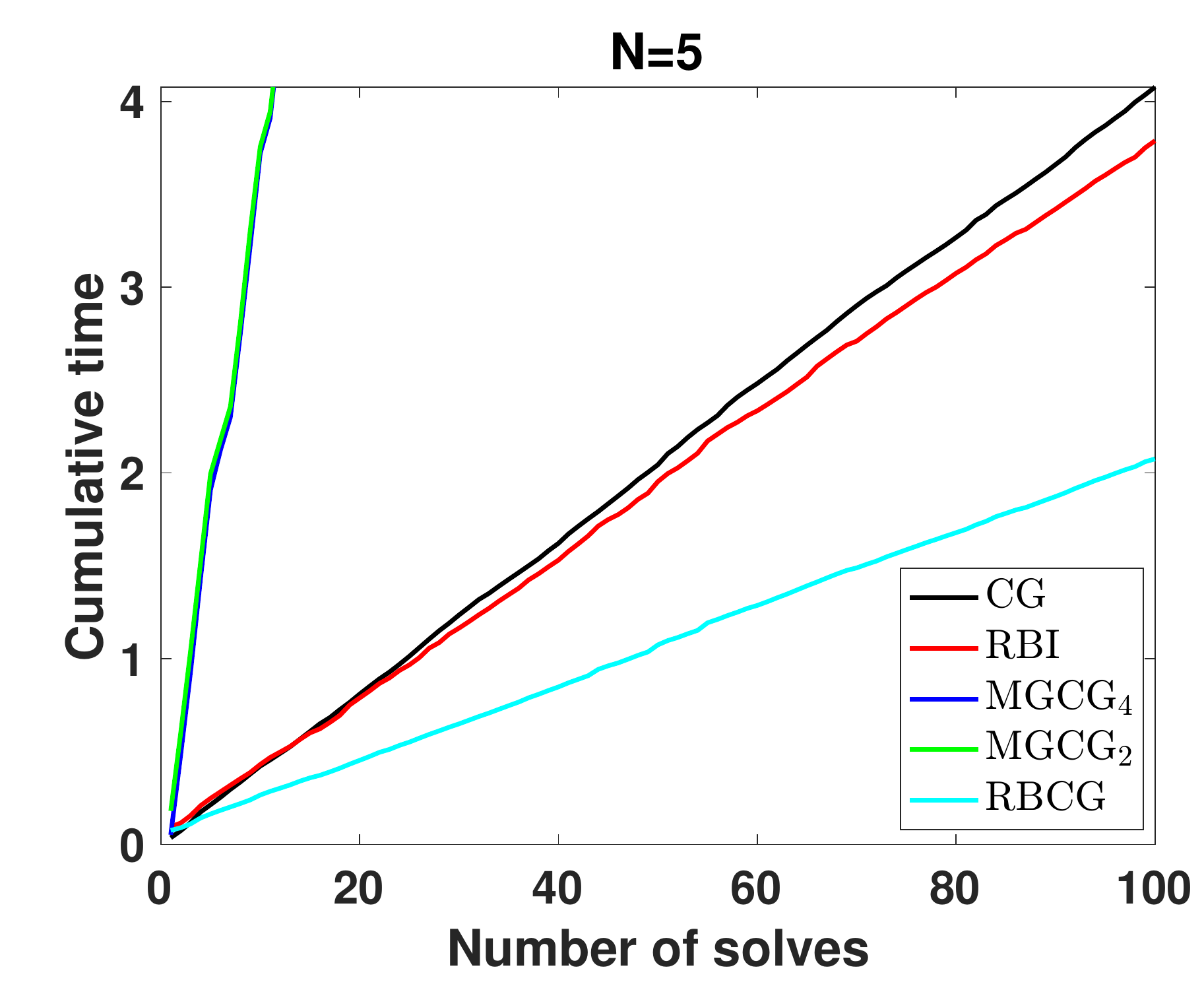}
\includegraphics[width=0.49\textwidth]{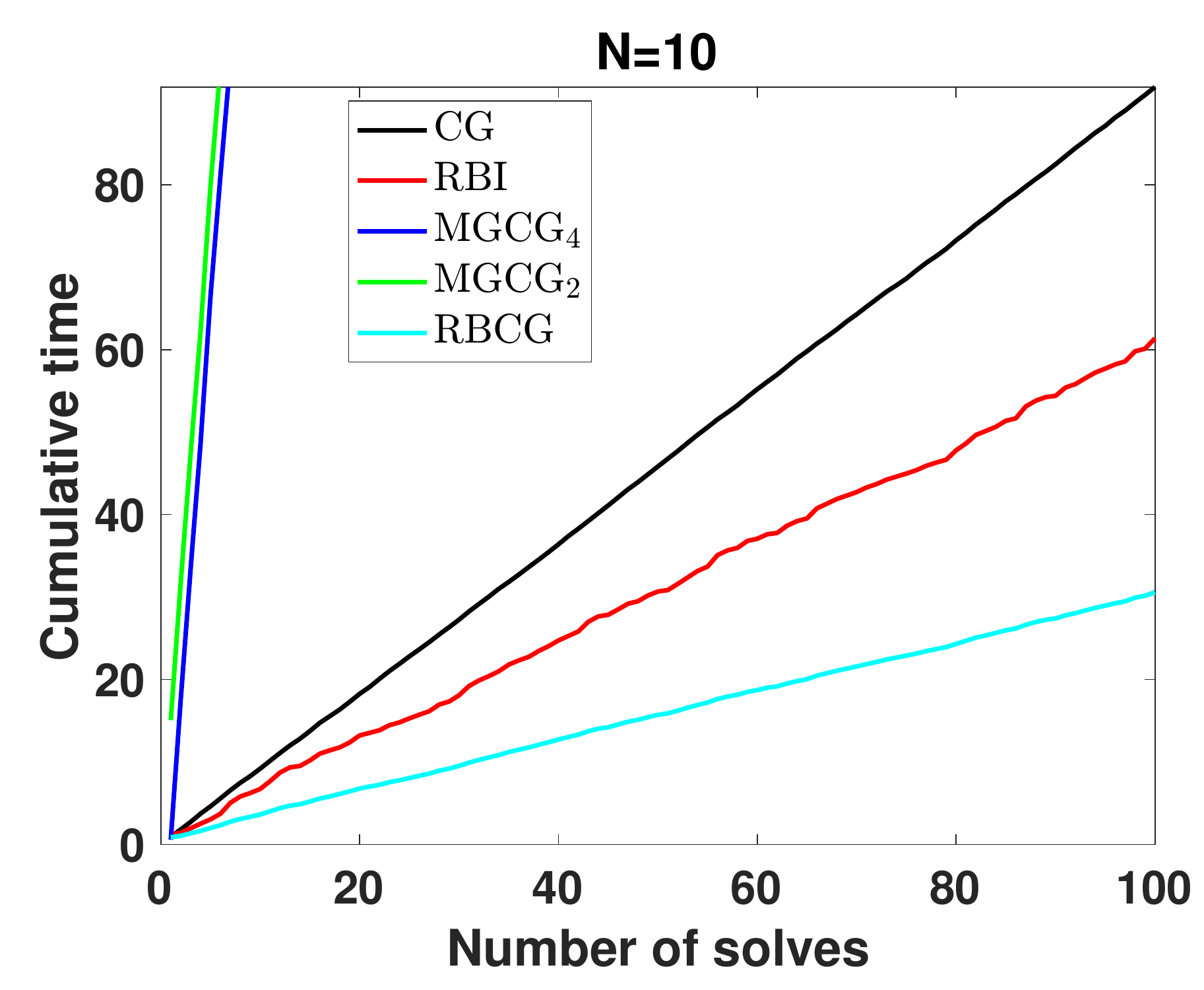}
\end{center}
\caption{The time as a function of the number of linear system solves for each scheme. Pictured on the left are for Case 1, and right are for Case 2.}
\label{fig:time}
\end{figure}

\subsection{Impact of the RB dimension on convergence}

Here, we would like to examine the impact of the RB dimension on the convergence of the RBI and RBCG schemes. Toward that end, we plot the average of the numbers of steps (defined to be the average of the ``best'' and ``worst'' cases shown in Figures \ref{fig:conv_case1} and \ref{fig:conv_case2}) as a function of the RB dimension. This is displayed in Figure \ref{fig:RBdimImpact}. We observe that the increase of the RB dimension clearly has a positive impact on the convergence of the schemes.

\begin{figure}
\begin{center}
\includegraphics[width=0.49\textwidth]{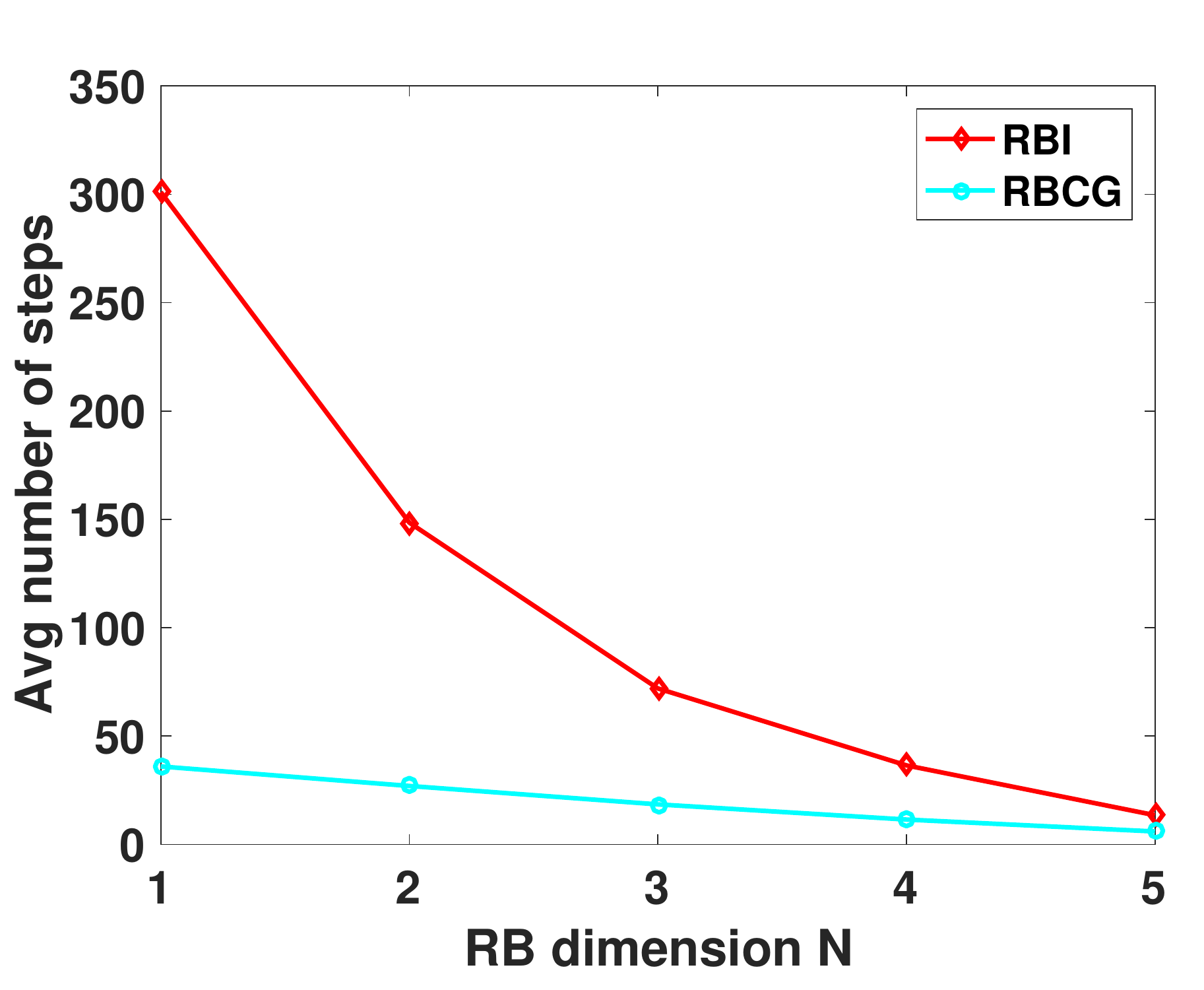}
\includegraphics[width=0.49\textwidth]{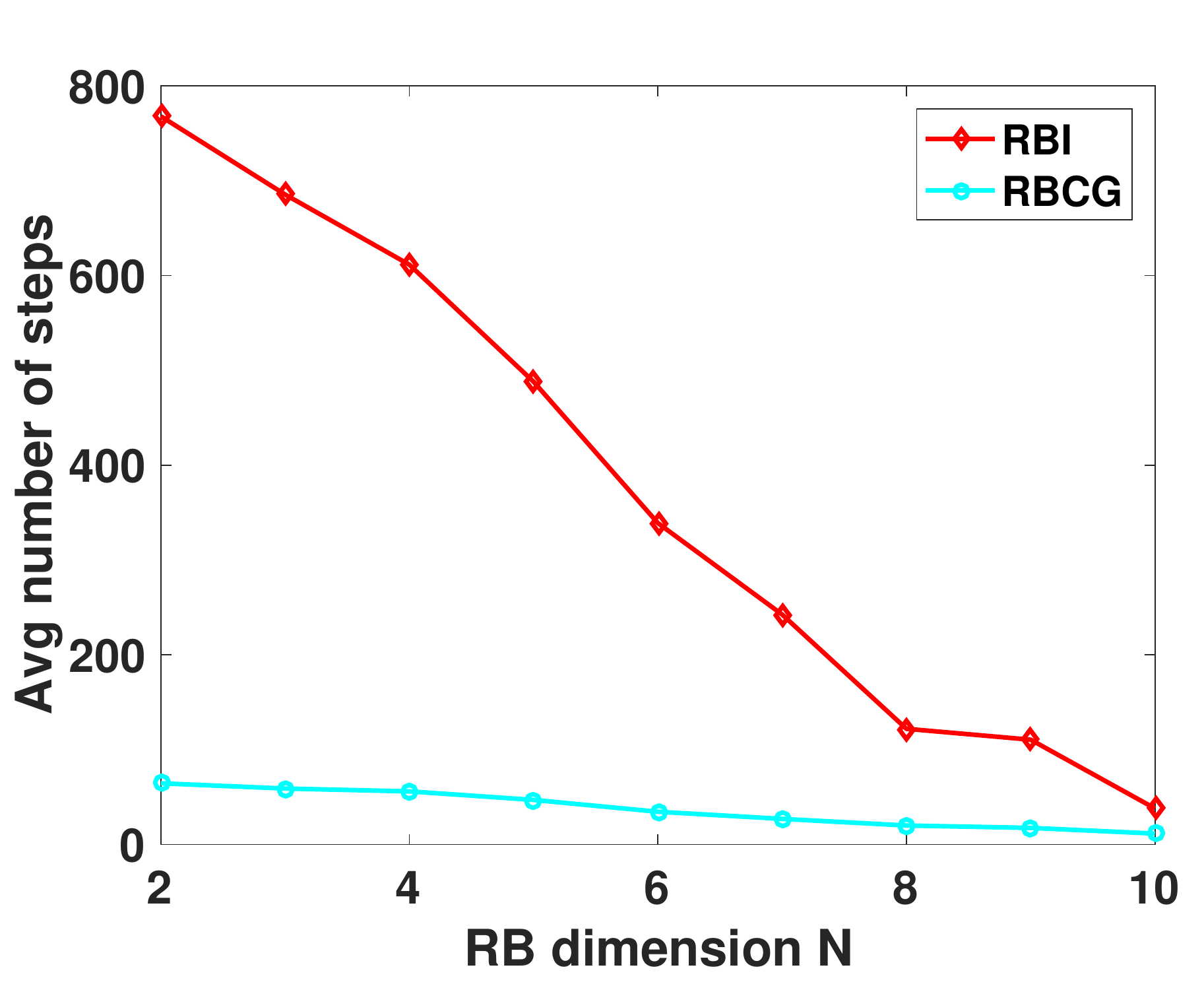}
\end{center}
\caption{The number of steps toward convergence for the two RB-based schemes for Case 1 (left) and Case 2 (right).}
\label{fig:RBdimImpact}
\end{figure}

\subsection{Performance for large system}

\rev{Finally, we test the methods on a large system. Toward that end, we refine the mesh of case $1$ until there are more than $2$ million degrees of freedom. The results are given in Figure \ref{fig:largesystem}. We see that the performance of RBCG is still comparable with the multigrid preconditioned CG in terms of convergence. However, for a large system like this, more levels are necessary for the V-cycle of multigrid. The situation is exacerbated by the storage and application of the multiple large restriction and prolongation operators, especially at the higher levels. As a result, it becomes challenging to apply MG-CG scheme on such systems, as can be seen from the computational time in Figure \ref{fig:largesystem} right. These constraints are substantially mitigated by our RBCG scheme.}

\begin{figure}
\begin{center}
\includegraphics[width=0.49\textwidth]{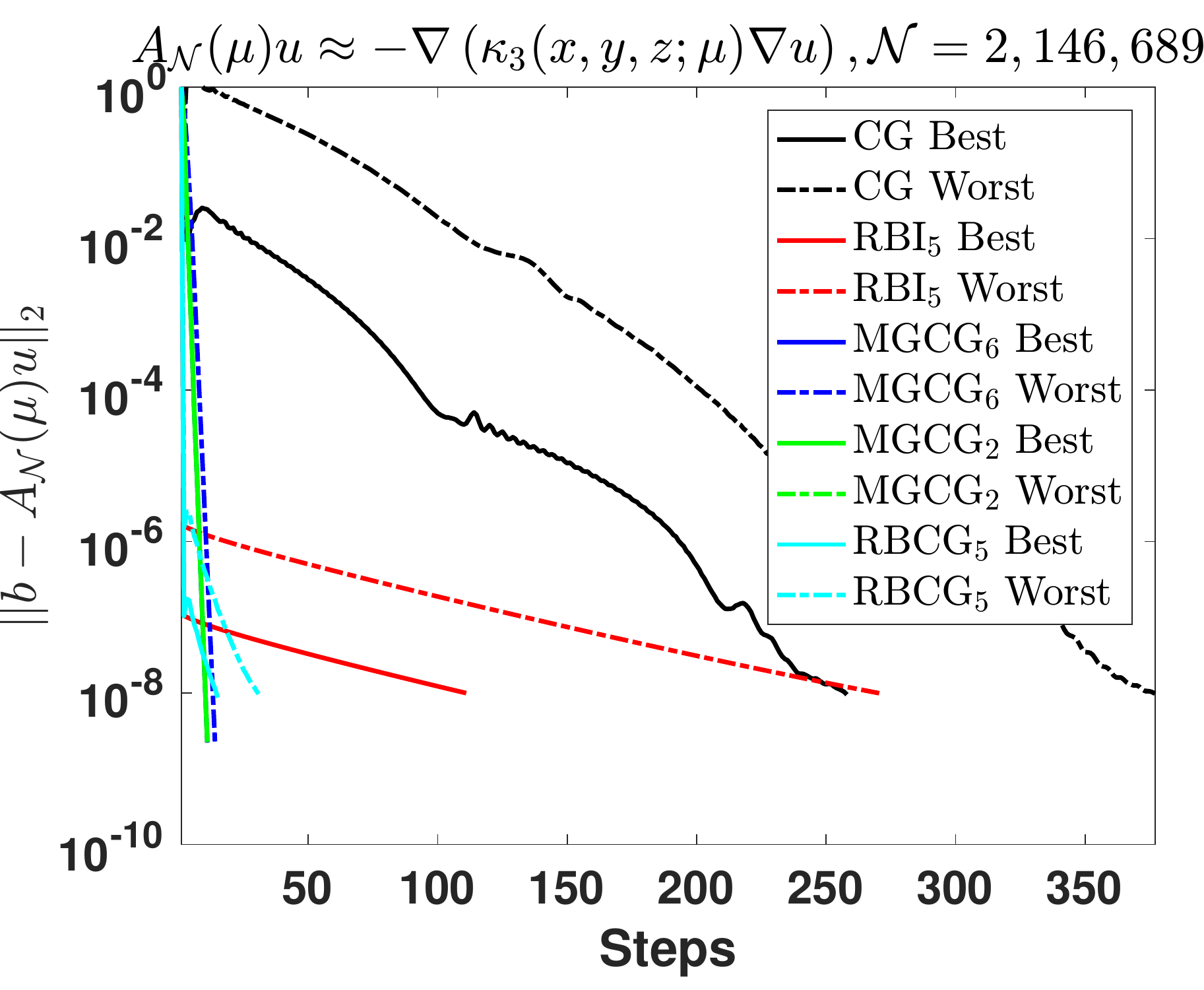}
\includegraphics[width=0.49\textwidth]{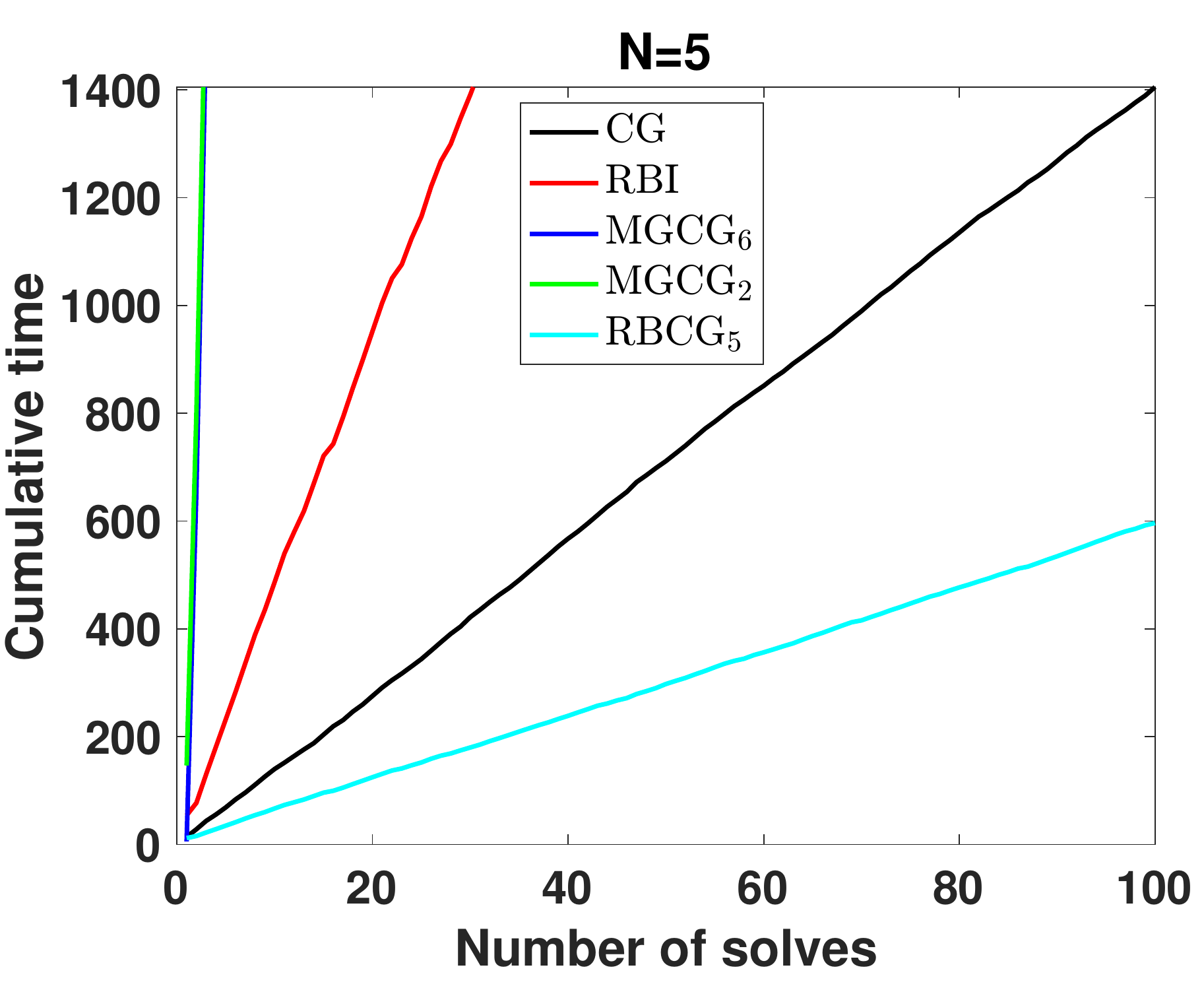}
\end{center}
\caption{\rev{Comparison of the histories of convergence and computation time when applied to a system of size more than $2$ million.}}
\label{fig:largesystem}
\end{figure}

\section{Conclusion}

\rev{Inspired by the traditional RB and multigrid methods, we develop and demonstrate a class of reduced basis methods, the first of its kind, for iteratively solving parametrized SPD linear systems.  It employs a greedy algorithm to efficiently determine sampling parameters and associated basis vectors and build a system-specific subspace.  Based on this subspace, solution procedures are designed to employ the reduced basis approximation as a {\em stand-alone iterative solver} or as a {\em preconditioner} in the conjugate gradient method. Numerical experiments demonstrate the superiority of the methods both in terms of number of steps toward convergence and the computational cost of each step.}

\rev{Extension of this novel approach to the non-symmetric and indefinite cases, and to systems resulting from time-dependent problems are a subject of ongoing research. For the RB schemes, the main role of the smoother is to reduce the error rather than smooth it. Hence, the construction of an appropriate smoother for the RB schemes is also worthy of investigation.}



\bibliographystyle{abbrv}

\end{document}